\documentclass[11pt]{article}
\usepackage{fullpage, amsmath}
\usepackage{amsthm,amsfonts,url,amssymb}
\usepackage{mathrsfs}
\usepackage{amssymb,amscd}
\usepackage{color}
\usepackage{chemarrow}
\usepackage{pictexwd,dcpic}
\usepackage{extarrows}
\usepackage[colorlinks,linkcolor=red,anchorcolor=green,citecolor=blue]{hyperref}
\usepackage{enumitem}
\usepackage{lipsum}
\usepackage[all,cmtip]{xy}
\usepackage{leftidx}
\usepackage[version=3]{mhchem}
\usepackage{imakeidx}

\newtheorem{theorem}{Theorem}[section]
\newtheorem{corollary}[theorem]{Corollary}
\newtheorem{lemma}[theorem]{Lemma}
\newtheorem{proposition}[theorem]{Proposition}

\newtheorem{hypothesis}[theorem]{Hypothesis}
\newtheorem{remark}[theorem]{Remark}

\newcommand{\hooklongrightarrow}{\lhook\joinrel\longrightarrow}
\newcommand{\twoheadlongrightarrow}{\relbar\joinrel\twoheadrightarrow}

\newcommand{\ra}{\rightarrow}
\newcommand{\lra}{\longrightarrow}

\newcommand{\ul}{\underline}

\newcommand{\bA}{\mathbb A}

\newcommand{\Q}{\mathbb Q}

\newcommand{\cN}{\mathcal N}

\newcommand{\co}{\mathcal O}

\newcommand{\cR}{\mathcal R}

\newcommand{\cS}{\mathcal S}

\newcommand{\cI}{\mathcal I}

\newcommand{\cM}{\mathcal M}
\newcommand{\cF}{\mathcal F}

\newcommand{\cE}{\mathcal E}

\newcommand{\cJ}{\mathcal J}

\newcommand{\cU}{\mathcal U}

\newcommand{\fm}{\mathfrak{m}}
\newcommand{\ub}{\mathfrak b}

\newcommand{\fp}{\mathfrak p}

\newcommand{\sC}{\mathscr C}

\newcommand{\sF}{\mathscr F}
\newcommand{\sI}{\mathscr I}
\newcommand{\sG}{\mathscr G}
\newcommand{\sE}{\mathscr E}

\newcommand{\lin}{\rule[2.5pt]{10pt}{0.5 pt}}

\DeclareMathOperator{\gl}{\mathfrak gl}

\DeclareMathOperator{\GL}{\mathrm GL}

\DeclareMathOperator{\Fil}{\mathrm Fil}

\DeclareMathOperator{\Gal}{\mathrm Gal}
\DeclareMathOperator{\Hom}{\mathrm Hom}

\DeclareMathOperator{\cris}{\mathrm cris}
\DeclareMathOperator{\rig}{\mathrm rig}
\DeclareMathOperator{\an}{\mathrm an}
\DeclareMathOperator{\Spec}{\mathrm Spec}

\DeclareMathOperator{\dR}{\mathrm dR}

\DeclareMathOperator{\Ind}{\mathrm Ind}
\DeclareMathOperator{\unr}{\mathrm unr}

\DeclareMathOperator{\Ker}{\mathrm Ker}
\DeclareMathOperator{\pr}{\mathrm pr}

\DeclareMathOperator{\Ext}{\mathrm Ext}

\DeclareMathOperator{\Spf}{\mathrm Spf}
\DeclareMathOperator{\Ima}{\mathrm Im}

\DeclareMathOperator{\lalg}{\mathrm lalg}

\DeclareMathOperator{\id}{\mathrm id}

\DeclareMathOperator{\dett}{\mathrm det}
\DeclareMathOperator{\alg}{\mathrm alg}

\DeclareMathOperator{\soc}{\mathrm soc}

\DeclareMathOperator{\diag}{\mathrm diag}

\DeclareMathOperator{\fss}{\mathrm fs}

\DeclareMathOperator{\sm}{\mathrm sm}
\DeclareMathOperator{\tri}{\mathrm tri}
\DeclareMathOperator{\pLL}{\mathrm pLL}

\DeclareMathOperator{\univ}{\mathrm univ}
\DeclareMathOperator{\rec}{\mathrm rec}
\begin{document}
	\title{$p$-adic Hodge parameters in the crystabelline representations of $\GL_3(\Q_p)$}
	\author{Yiwen Ding}
	\date{}
	\maketitle
	\begin{abstract}We build a one-to-one correspondence between  $3$-dimensional (generic) crystabelline representations  of the absolute Galois group of $\Q_p$ and certain locally analytic representations of $\GL_3(\Q_p)$. We show that the correspondence can be realized in subspaces of $p$-adic automorphic representations.
	\end{abstract}
	\tableofcontents
	
	\section{Introduction}
	The locally analytic $p$-adic Langlands correspondence  for $\GL_2(\Q_p)$ builds a one-to-one correspondence betwen two dimensional $p$-adic representations of the absolute Galois group $\Gal_{\Q_p}$ of $\Q_p$, and  certain locally analytic representations of $\GL_2(\Q_p)$ (e.g. see \cite[Thm.~0.1]{Colm16}).
	In particular, the correspondence provides a  dictionary of parameters on the two  sides.  In the non-split crystalline case, the  correspondence is rather straightforward, since a $2$-dimensional crystalline representation $\rho$ of $\Gal_{\Q_p}$ is simply determined by its associated Weil-Deligne representation and Hodge-Tate weights. Correspondingly, the locally analytic representation $\pi(\rho)$ associated to $\rho$ is determined by its locally algebraic subrepresentation, as characterized by the classical local Langlands correspondence. This phenomenon collapses for other groups (say $\GL_n(\Q_p)$, $n>2$, or $\GL_2(K)$, $K\neq \Q_p$), where  additional parameters of the Hodge filtration appear in the crystalline case. An immediate question in the $p$-adic Langlands program for other groups is  how to see such new parameters on the automorphic side. In the note, we address this question for $\GL_3(\Q_p)$. 
	
	We introduce some notation. Let $\rho$ be a $3$-dimensional crystalline representation of $\Gal_{\Q_p}$ over a sufficiently large $p$-adic field $E$, of regular Hodge-Tate weights $h:=h_1>h_2>h_3$. Let   $\alpha_1, \alpha_2, \alpha_3$ be the eigenvalues of the crystalline Frobenius on $D_{\cris}(\rho)$. We assume $\alpha_i \alpha_j^{-1} \neq 1, p$ for $i \neq j$ (this is what we mean $\rho$ is generic).  Let $e_i$ be an $\alpha_i$-eigenvector for $\varphi$.  We furthermore assume $\rho$ is non-critical for all refinements, which means the Hodge filtration of $D_{\cris}(\rho)$ is in a relative general position with all the $\varphi$-stable flags: $E e_i \subsetneq E e_i \oplus E e_j \subsetneq D_{\cris}(\rho)$. Multiplying $e_i$ by certain non-zero scalars, the Hodge filtration of $D_{\cris}(\rho)$ can be explicitly described as follows:
	\begin{equation}\label{Ead}
		\Fil^j D_{\cris}(\rho)=\begin{cases}
			D_{\cris}(\rho) & j\leq -h_1 \\
			E(e_1+e_2) \oplus E(e_1+a_{\rho} e_2+e_3) & -h_1<j \leq -h_2 \\
			E(e_1+a_{\rho} e_2+e_3) & -h_2<j \leq -h_3 \\
			0 & j>-h_3
		\end{cases}
	\end{equation}
	where $a_{\rho}\in E^{\times}$ is what we call the Hodge parameter of $\rho$. We denote such $\rho$ by $V(\ul{\alpha}, h, a_{\rho})$. The set of non-critical crystalline  representations of weights $h$ and $\varphi$-eigenvalues $\ul{\alpha}=(\alpha_1, \alpha_2, \alpha_3)$ is exactly $\{V(\ul{\alpha}, h, a)\}_{a\in E^{\times}}$.
	In this note, for $a\in E^{\times}$, we associate a locally analytic representation $\pi(\ul{\alpha}, h, a)$ of $\GL_3(\Q_p)$. Here is our main theorem.
	\begin{theorem}\label{Tmain}
		(1) (Local correspondence) We have $\pi(\ul{\alpha}, h,a)\cong \pi(\ul{\alpha},h,a')$ if and only if $a=a'$ if and only if $V(\ul{\alpha}, h, a)\cong V(\ul{\alpha}, h, a')$.
		
		(2) (Local-global compatibility) Assume $\rho=V(\ul{\alpha}, h,a_{\rho})$ is automorphic for the setting of \cite{CEGGPS1}, and let $\widehat{\pi}(\rho)$ be the (globally) associated Banach representation of $\GL_3(\Q_p)$. Then for $a\in E^{\times}$,
		\begin{equation*}
			\pi(\ul{\alpha}, h, a)\hookrightarrow \widehat{\pi}(\rho) \text{ if and only if } a=a_{\rho}.
		\end{equation*}
	\end{theorem}
	\begin{remark}\label{Rinto}
		(1)		Let $\Pi_{\infty}$ be the patched Banach representation  over the patched Galois deformation ring $R_{\infty}$  (for a certain $3$-dimensional mod $p$ representation $\overline{\rho}$ of $\Gal_{\Q_p}$) of \cite{CEGGPS1}. Our local-global compatibility result is in fact obtained for a more general setting. We  show that if there is a maximal ideal $\fm$ of $R_{\infty}[1/p]$ associated to $\rho$ (which implies $\rho$ has mod $p$ reduction isomorphic to $\overline{\rho}$) such that 
		$\Pi_{\infty}[\fm]^{\lalg} \neq 0$, then for $a\in E^{\times}$,
		\begin{equation*}
			\pi(\ul{\alpha}, h, a)\hookrightarrow \Pi_{\infty}[\fm] \text{ if and only if } a=a_{\rho}.
		\end{equation*}
		The assumption $\Pi_{\infty}[\fm]^{\lalg} \neq 0$ is equivalent to that $\rho$ appears on the patched eigenvariety associated to $\Pi_{\infty}$ (where the equivalence  is easy in our non-critical case, but noting it also holds in the critical case by \cite{BHS2} \cite{BHS3}).
		
		(2) For the case where $\rho$ is critical for some refinements or if $\rho$ has irregular Sen weights,  one can  associate to $\rho$ a semi-simple locally analytic representation of $\GL_3(\Q_p)$ that determines $\rho$ as in \cite{Br13I} \cite{Wu21}. Together with our construction of $\pi(\ul{\alpha}, h, a)$, this establishes a one-to-one correspondence between $3$-dimensional generic crystabelline $\Gal_{\Q_p}$-representations and their corresponding locally analytic representations of $\GL_3(\Q_p)$.
		
		(3) Let $\pi_{\alg}(\ul{\alpha}, h)$ be the locally algebraic representation of $\GL_3(\Q_p)$ associated to $\ul{\alpha}$ and $h$ (which is the $\pi_{\alg}(\ul{\alpha}, \lambda)$ in (\ref{Eintinj})). A key feature of $\pi(\ul{\alpha}, h, a)$ is that it contains two copies of $\pi_{\alg}(\ul{\alpha}, h)$, one in the socle, another one in the cosocle (see (\ref{Eipic1})):
		\begin{equation*}
			\pi_{\alg}(\ul{\alpha}, h)\cong \soc_{\GL_3(\Q_p)}  \pi(\ul{\alpha}, h, a) \hooklongrightarrow \pi(\ul{\alpha}, h, a) \buildrel{\kappa}\over\twoheadlongrightarrow \mathrm{cosoc}_{\GL_3(\Q_p)} \pi(\ul{\alpha}, h,a)\cong \pi_{\alg}(\ul{\alpha}, h).
		\end{equation*}
		The Hodge parameter $a\in E^{\times}$ is encoded in the extension of (the cosocle) $\pi_{\alg}(\ul{\alpha}, h)$ by $\Ker(\kappa)$. A similar pattern of extensions (with locally algebraic representations in the cosocle) has appeared in the semi-stable non-crystalline case (cf. \cite{Br16}\cite{BD1}, tracing back to \cite{Br04}). This work grows out of the finding of the ``surplus" locally algebraic constituents in the non-critical case in \cite[\S~Appendix C]{Ding15}. Note that the existence of such a constituent was first proved by Hellmann-Hernandez-Schraen in the split case for $\GL_3(\Q_p)$.

		(4) The representation $\pi(\ul{\alpha},h,a_{\rho})$ is a proper subrepresentation of $\widehat{\pi}(\rho)^{\an}$ (the locally analytic subrepresentation of $\widehat{\pi}(\rho)$), which has many more constituents. In fact, some of these constituents are already described in \cite[\S~5.3]{BH2}. We remark that even when amalgamating our representation $\pi(\ul{\alpha},h,a_{\rho})$ with $\pi(\rho)^{\fss}$ of \cite{BH2}, the result  should still give only a proper subrepresentation of $\widehat{\pi}(\rho)^{\an}$. For example, one  may expect that $\widehat{\pi}(\rho)^{\an}$ contains some (mysterious) supersingular constituent(s) (cf. \cite{Ding15}).

		(5) We finally remark the results in the note are obtained for (more general) crystabelline representations.
	\end{remark}
	We explain the construction of $\pi(\ul{\alpha}, h, a)$ and the proof of Theorem \ref{Tmain}. We first give a reinterpretation of the Hodge parameter. Let $D:=D_{\rig}(\rho)$  be the associated $(\varphi, \Gamma)$-module of rank $3$ over the Robba ring $\cR_E$. Recall an ordering of $\ul{\alpha}$ corresponds to a non-critical triangulation of $D$.  Let $D_1$ (resp. $C_1$) be the rank $2$ saturated $(\varphi, \Gamma)$-submodule (resp. quotient)  of $D$ corresponding to $\alpha_1$ and $\alpha_2$. So $D_1$ and $C_1$ depend on the choice of two $\varphi$-eigenvalues, but in fact any choice will work.  Then $D$ has the following two forms (where a line means a non-split extension with the left object the sub, and the right the quotient)
	\begin{equation*}
		\Big[\underbrace{\cR_E(\unr(\alpha_1)z^{h_1}) \lin \cR_E(\unr(\alpha_2) z^{h_2})}_{D_1} \lin \cR_E(\unr(\alpha_3) z^{h_3})\Big]
	\end{equation*}
	\begin{equation*}
		\Big[ \cR_E(\unr(\alpha_3) z^{h_1})\lin\underbrace{\cR_E(\unr(\alpha_1)z^{h_2}) \lin \cR_E(\unr(\alpha_2) z^{h_3})}_{C_1} \Big].
	\end{equation*}
	Let $\iota_D\in \Hom_{(\varphi, \Gamma)}(D_1,C_1)$ be the natural (injective) composition
	\begin{equation*}
		\iota_D: 	D_1 \hooklongrightarrow D \twoheadlongrightarrow C_1.
	\end{equation*}
	\begin{proposition}\label{Pintro}
		We have $\dim_E \Hom_{(\varphi, \Gamma)}(D_1,C_1)=2$, and $D$ is determined by the line $E[\iota_D]\subset \Hom_{(\varphi, \Gamma)}(D_1,C_1)$ (together with $\alpha_3$).
	\end{proposition}
	The $E$-vector space $\Hom_{(\varphi, \Gamma)}(D_1,C_1)$ is in fact spanned by the two (non-injective) maps: $D_1 \twoheadrightarrow \cR_E(\unr(\alpha_i)z^{h_2})\hookrightarrow C_1$ for $i=1,2$. We give a quick explanation why it is the right parameter space for $a_{\rho}$ in (\ref{Ead}). An injection $\iota: D_1 \hookrightarrow C_1$ induces a map of filtered $\varphi$-modules $\iota: D_{\cris}(D_1) \ra D_{\cris}(C_1)$. Denoting by $\Fil^{\max}$ the unique non-trivial Hodge filtration for $D_{\cris}(D_1)$ and $D_{\cris}(C_1)$, the map $\iota$ then carries exactly the information on the relative position of the two lines $\iota(\Fil^{\max} D_{\cris}(D_1))$ and $\Fil^{\max} D_{\cris}(C_1)$ in $D_{\cris}(C_1)$. When $[\iota]\in E[\iota_D]$, we see the parameter for the relative position is just $a_{\rho}$ (see Remark \ref{Rfphi} for more details). 
	
	Let $\iota\in \Hom_{(\varphi, \Gamma)}(D_1, C_1)$ be an injection. We first associate to $\iota$ a set of pairs of certain deformations of $D_1$ and $C_1$, which will be crucially used in our construction of $\pi(\ul{\alpha}, h, a)$. Let $\sI_{\iota}$ be the set of pairs $(\widetilde{D}_1, \widetilde{C}_1)$, that we call \textit{higher intertwining pairs} (see Theorem \ref{Tint2} (2) below), satisfying 
	\begin{itemize}
		\item $\widetilde{D}_1 \in \iota^-(\Ext^1_{(\varphi, \Gamma)}(C_1,D_1))\subset \Ext^1_{(\varphi, \Gamma)}(D_1,D_1)$ and $\widetilde{C}_1\in  \iota^+(\Ext^1_{(\varphi, \Gamma)}(C_1,D_1))\subset \Ext^1_{(\varphi, \Gamma)}(C_1, C_1)$, where $\iota^-$ and $\iota^+$ denote the natural pull-back and push-forward maps respectively,
		\item there exists $\tilde{\iota}: \widetilde{D}_1 \hookrightarrow \widetilde{C}_1$ such that the following diagram commutes
		\begin{equation*}\begin{CD}
				0 @>>> D_1 @>>> \widetilde{D}_1 @>>> D_1 @>>> 0 \\
				@. @V\iota VV @V\tilde{\iota} VV @V \iota VV \\
				0 @>>> C_1 @>>> \widetilde{C}_1 @>>> C_1 @>>> 0.
			\end{CD}
		\end{equation*}
	\end{itemize}
The following proposition summarizes some properties of $ \iota^-(\Ext^1_{(\varphi, \Gamma)}(C_1,D_1))$ and $ \iota^+(\Ext^1_{(\varphi, \Gamma)}(C_1,D_1))$.
\begin{proposition}
	 For an injection $\iota\in \Hom_{(\varphi, \Gamma)}(D_1,C_1)$, $$\dim_E\iota^-(\Ext^1_{(\varphi, \Gamma)}(C_1, D_1))=\dim_E \iota^+(\Ext^1_{(\varphi, \Gamma)}(C_1,D_1))=3.$$ Moreover, $\iota^-(\Ext^1_{(\varphi, \Gamma)}(C_1,D_1)) \supset \Ext^1_g(D_1,D_1)$ (resp. $\iota^+(\Ext^1_{(\varphi, \Gamma)}(C_1,D_1))\supset \Ext^1_g(C_1,C_1)$), the subspace of de Rham deformations, and any trianguline deformation contained in 	$\iota^-(\Ext^1_{(\varphi, \Gamma)}(C_1,D_1))$	(resp. $\iota^+(\Ext^1_{(\varphi, \Gamma)}(C_1,D_1))$) is de Rham.
\end{proposition}
The motivation to study such pairs is from the following theorem.
	\begin{theorem}\label{Tint2}
		 (1) For injections $\iota$, $\iota'\in \Hom_{(\varphi, \Gamma)}(D_1, C_1)$, $[\iota]\in E[\iota']$ if and only if $\sI_{\iota}=\sI_{\iota'}$.
		
		(2) Let $D$ be given as above Proposition \ref{Pintro}, then $\sI_{\iota_D}$ is equal to the set of pairs $(\widetilde{D}_1, \widetilde{C_1})\in \Ext^1_{(\varphi, \Gamma)}(D_1, D_1) \times \Ext^1_{(\varphi, \Gamma)}(C_1,C_1)$ such that there exists a deformation $\widetilde{D}$ of $D$ over $\cR_{E[\epsilon]/\epsilon^2}$ which sits in both of the following two exact sequences (of $(\varphi, \Gamma)$-modules over $\cR_{E[\epsilon]/\epsilon^2}$)
		\begin{equation*}
			0 \lra \widetilde{D}_1 \lra \widetilde{D} \lra \cR_{E[\epsilon]/\epsilon^2}(\unr(\alpha_3)z^{h_3}) \lra 0,
		\end{equation*}
		\begin{equation*}
			0 \lra\cR_{E[\epsilon]/\epsilon^2}(\unr(\alpha_3)z^{h_1})  \lra \widetilde{D} \lra \widetilde{C}_1 \lra 0.
		\end{equation*}
	\end{theorem}
	In fact, we show that $\iota$ can be detected by a single element $(\widetilde{D}_1, \widetilde{C}_1)\in \sI_{\iota}$ with $\widetilde{D}_1$ \textit{non-trianguline}. 
	
	Now we move to the $\GL_3(\Q_p)$-side. We introduce a bit more notation.  Let $\lambda:=(h_1-2,h_2-1,h_3)$, $L(\lambda)$ be the algebraic representation of $\GL_3(\Q_p)$ of highest weight $\lambda$. Let $\pi_{\alg}(\ul{\alpha}, \lambda):=(\Ind_{B^-}^{\GL_3} \unr(p^2\alpha_1) \boxtimes \unr(p \alpha_2) \boxtimes \unr(\alpha_3))^{\infty} \otimes_E L(\lambda)$, which is no other than the locally algebraic representation associated to $D$. By the classical intertwining property, the smooth induction remains unchanged if one alters the ordering of $\alpha_i$. We can also consider the locally analytic parabolic induction, and we have
	\begin{equation}\label{Eintinj}
		\pi_{\alg}(\ul{\alpha}, \lambda) \hooklongrightarrow \big(\Ind_{B^-}^{\GL_3} \unr(p^2\alpha_1)z^{h_1-2}\boxtimes \unr(p \alpha_2) z^{h_2-1}\boxtimes \unr(\alpha_3)z^{h_3}\big)^{\an}=: I(\ul{\alpha}, \lambda).
	\end{equation}
	The representation $I(\ul{\alpha}, \lambda)$ now depends on the ordering of $\ul{\alpha}$.  For $w$ running through $S_3$, we obtain six locally analytic principal series $I(w(\ul{\alpha}), \lambda)$. The structure of $I(w(\ul{\alpha}), \lambda)$ is clear by \cite{OS} (which has the same pattern as the dual Verma module). For our application, we are mostly interested in the constituents \textit{right after} the locally algebraic $\pi_{\alg}(\ul{\alpha}, \lambda)$: the socle of $I(w(\ul{\alpha}), \lambda)/\pi_{\alg}(\ul{\alpha},\lambda)$ has the form $\sC(s_1,w) \oplus \sC(s_2,w)$ (which are distinct and topologically irreducible). We refer to \S~\ref{S3.1} for the precise definition. These representations are not all distinct when varying $w$: $\sC(s_i, w)\cong \sC(s_j, w')$ if and only if $s_i=s_j$ and $w'\in \{w, s_k w\}$ with $s_k\neq s_i$. Let $\cS:=\{\sC(s_i,w)\}$, where $w$ has minimal length in $\{w,s_k w\}$ for $s_k\neq s_i$. So $\sharp \cS=6$. By amalgamating the six $I(w(\ul{\alpha}), \lambda)$ as much as possible, the resulting representation contains a unique subrepresentation $\pi_1(\ul{\alpha}, \lambda)$ of the form $[\pi_{\alg}(\ul{\alpha},\lambda) \lin \oplus_{\sC\in \cS} \sC]$  (with socle $\pi_{\alg}(\ul{\alpha}, \lambda)$).
	
	We discuss a bit more on parabolic inductions. Let $\pi(D_1)$ be the locally analytic representation of $\GL_2(\Q_p)$ associated to $D_1$, $\pi_{\alg}(D_1)$ be its locally algebraic vector. Consider the parabolic induction $\big(\Ind_{P_1^-}^{\GL_3} (\pi(D_1) \otimes_E \varepsilon^{-1} \circ \dett) \boxtimes \unr(\alpha_3)z^{h_3}\big)^{\an}$, whose socle is $\pi_{\alg}(\ul{\alpha}, \lambda)$. Again we just look at the constituents right after the locally algebraic $\pi_{\alg}(\ul{\alpha}, \lambda)$, and it turns out that we encounter exactly three of those in $\cS$: $\{\sC(s_1, 1), \sC(s_1,s_1), \sC(s_2, 1)\}=:\cS^-$. And the parabolic induction contains a unique subrepresentation $\pi_1(\ul{\alpha},\lambda)^-$ of the form $[\pi_{\alg}(\ul{\alpha}, \lambda) \lin \oplus_{\sC\in \cS^-} \sC]$. Replacing $D_1$ by $C_1$ and $P_1$ by $P_2$, and by a similar discussion, $(\Ind_{P_2^-}^{\GL_3} \unr(\alpha_3)z^{h_1} \varepsilon^{-2} \boxtimes \pi(C_1))^{\an}$ gives exactly the other three constituents $\cS^+:=\{\sC(s_2,s_2), \sC(s_2,s_2s_1), \sC(s_1, s_1s_2)\}$, and contains a unique representation $\pi_1(\ul{\alpha}, \lambda)^+$ of the form $[\pi_{\alg}(\ul{\alpha}, \lambda) \lin \oplus_{\sC\in \cS^+} \sC]$. We have hence 
	\begin{equation*}
		\pi_1(\ul{\alpha}, \lambda) \cong \pi_1(\ul{\alpha}, \lambda)^- \oplus_{\pi_{\alg}(\ul{\alpha},\lambda)} \pi_1(\ul{\alpha}, \lambda)^+.
	\end{equation*}
	Using the $p$-adic Langlands correspondence for $\GL_2(\Q_p)$ and taking the corresponding parabolic inductions, we can obtain two compositions (which are injective)
	\begin{multline*}
		j^-:	\Ext^1_{(\varphi, \Gamma)}(D_1,D_1) \xlongrightarrow{\pLL} \Ext^1_{\GL_2(\Q_p)}(\pi_{\alg}(D_1),\pi(D_1))  \\
		\xlongrightarrow{\Ind}\Ext^1_{\GL_3(\Q_p)}(\pi_{\alg}(\ul{\alpha}, \lambda), \pi_1(\ul{\alpha},\lambda)^-)\hooklongrightarrow \Ext^1_{\GL_3(\Q_p)}(\pi_{\alg}(\ul{\alpha}, \lambda), \pi_1(\ul{\alpha},\lambda)),
	\end{multline*}
	\begin{multline*}
		j^+:	\Ext^1_{(\varphi, \Gamma)}(C_1,C_1) \xlongrightarrow{\pLL} \Ext^1_{\GL_2(\Q_p)}(\pi_{\alg}(C_1),\pi(C_1))  \\
		\xlongrightarrow{\Ind}\Ext^1_{\GL_3(\Q_p)}(\pi_{\alg}(\ul{\alpha}, \lambda), \pi_1(\ul{\alpha},\lambda)^+)\hooklongrightarrow \Ext^1_{\GL_3(\Q_p)}(\pi_{\alg}(\ul{\alpha}, \lambda), \pi_1(\ul{\alpha},\lambda)).
	\end{multline*}
	
	We can now give the definition of $\pi(\ul{\alpha}, h, a)$ in Theorem \ref{Tmain}, but we change the notation to $\pi(\ul{\alpha}, \lambda, \iota)$ to be consistent with the above discussion. Let $(\widetilde{D}_1, \widetilde{C}_1)\in \sI_{\iota}$ with $\widetilde{D}_1$ non-trianguline (which actually implies $\widetilde{C}_1$ non-trianguline either).  Consider $j^-(\widetilde{D}_1) \oplus_{\pi_1(\ul{\alpha}, \lambda)} j^+(\widetilde{C}_1)$, which is an extension of $\pi_{\alg}(\ul{\alpha}, \lambda)^{\oplus 2}$ by $\pi_1(\ul{\alpha}, \lambda)$ with the following form 
	\begin{equation}\label{Eipic2}
		\begindc{\commdiag}[100] 
		\obj(0,0)[a]{$\pi_{\alg}(\ul{\alpha},\lambda)$}
		\obj(8,-5)[b]{$\sC(s_1,1)$}
		\obj(8,-3)[c]{$\sC(s_1,s_1)$}
		\obj(8,-1)[d]{$\sC(s_2,1)$}
		\obj(8,1)[e]{$\sC(s_1, s_1s_2)$}
		\obj(8,3)[f]{$\sC(s_2,s_2)$}
		\obj(8,5)[g]{$\sC(s_2,s_2s_1)$}
		\obj(16,3)[i]{$\pi_{\alg}(\ul{\alpha},\lambda)$}
		\obj(16,-3)[h]{$\pi_{\alg}(\ul{\alpha},\lambda)$}
		\mor{a}{b}{}[+1,3]
		\mor{a}{c}{}[+1,3]
		\mor{a}{d}{}[+1,3]
		\mor{a}{e}{}[+1,3]
		\mor{a}{f}{}[+1,3]
		\mor{a}{g}{}[+1,3]
		\mor{h}{b}{}[+1,3]
		\mor{h}{c}{}[+1,3]
		\mor{h}{d}{}[+1,3]
		\mor{i}{e}{}[+1,3]
		\mor{i}{f}{}[+1,3]
		\mor{i}{g}{}[+1,3]
		\enddc
	\end{equation}
	Note that any  subextension of $\pi_{\alg}(\ul{\alpha}, \lambda)$ by $\pi_1(\ul{\alpha}, \lambda)$, of $j^-(\widetilde{D}_1) \oplus_{\pi_1(\ul{\alpha}, \lambda)} j^+(\widetilde{D}_1)$, is \textit{non-split}. It follows from the fact that  $j^-(\widetilde{D}_1)$ and $j^+(\widetilde{C}_1)$ are linearly independent for non-trianguline $\widetilde{D}_1$ and $\widetilde{C}_1$, simply because $\cS^-$ and $\cS^+$ are disjoint (as illustrated in (\ref{Eipic2})). Roughly speaking, this implies the Galois intertwining pair $(\widetilde{D}_1, \widetilde{C}_1)$ (for $\iota$) does \textit{not} intertwine on the automorphic side. 
	One can show that there is a unique extension, which is our $\pi(\ul{\alpha}, \lambda, \iota)$, of $\pi_{\alg}(\ul{\alpha}, \lambda)$ by $\pi_1(\ul{\alpha}, \lambda)$ satisfying that $\pi(\ul{\alpha}, \lambda, \iota)\subset j^-(\widetilde{D}_1') \oplus_{\pi_1(\ul{\alpha}, \lambda)} j^+(\widetilde{C}_1')$ for all $(\widetilde{D}_1', \widetilde{C}_1')\in \sI_{\iota}$ with $\widetilde{D}_1'$ non-trianguline. In fact, by suitably normalizing the maps $j^-$ and $j^+$, $\pi(\ul{\alpha},\lambda,\iota)$ can be defined to be the representation associated to $j^-(\iota^-([M]))-j^+(\iota^+([M]))$ for a non-de Rham $M\in \Ext^1_{(\varphi, \Gamma)}(C_1,D_1)$. Moreover, one can show  that $\pi(\ul{\alpha}, \lambda, \iota)$ determines $\sI_{\iota}$ hence $\iota$. This proves Theorem \ref{Tmain} (1). By construction, $\pi(\ul{\alpha}, \lambda, \iota)$ has the following form
	\begin{equation}\label{Eipic1}
		\begindc{\commdiag}[100] 
		\obj(0,0)[a]{$\pi_{\alg}(\ul{\alpha},\lambda)$}
\obj(8,-5)[b]{$\sC(s_1,1)$}
\obj(8,-3)[c]{$\sC(s_1,s_1)$}
\obj(8,-1)[d]{$\sC(s_2,1)$}
\obj(8,1)[e]{$\sC(s_1, s_1s_2)$}
\obj(8,3)[f]{$\sC(s_2,s_2)$}
\obj(8,5)[g]{$\sC(s_2,s_2s_1)$}
		\obj(16,0)[h]{$\pi_{\alg}(\ul{\alpha},\lambda)$}
		\mor{a}{b}{}[+1,3]
		\mor{a}{c}{}[+1,3]
		\mor{a}{d}{}[+1,3]
		\mor{a}{e}{}[+1,3]
		\mor{a}{f}{}[+1,3]
		\mor{a}{g}{}[+1,3]
		\mor{h}{b}{}[+1,3]
		\mor{h}{c}{}[+1,3]
		\mor{h}{d}{}[+1,3]
		\mor{h}{e}{}[+1,3]
		\mor{h}{f}{}[+1,3]
		\mor{h}{g}{}[+1,3]
		\enddc
	\end{equation}
	
	We discuss our local-global compatibility result. We use the setting in Remark \ref{Rinto} (1). Let $D:=D_{\rig}(\rho)$. We only explain why $\pi(\ul{\alpha}, \lambda, \iota_D)\hookrightarrow \Pi_{\infty}[\fm]$. To $(\widetilde{D}_1, \widetilde{C}_1)\in \sI_{\iota_D}$ with $\widetilde{D}_1$ non-trianguline,  Theorem \ref{Tint2} (2)  associates a deformation $\widetilde{D}$ of $D$. Let $\cI$ be a ``thickening" ideal of $\fm$ associated to $\widetilde{D}$, so there is an exact sequence
	\begin{equation}\label{Ethick}
		0 \lra \Pi_{\infty}[\fm] \lra \Pi_{\infty}[\cI] \lra \Pi_{\infty}[\fm].
	\end{equation}
	Based on some local-global compatibility results on the Jacquet-Emerton modules of $\Pi_{\infty}^{R_{\infty}-\an}[\cI]$, one can show the two presentations of $\widetilde{D}$ in Theorem \ref{Tint2} (2) give rise to two  injections $j^-(\widetilde{D}_1)\hookrightarrow \Pi_{\infty}[\cI]$ and $j^+(\widetilde{D}_1)\hookrightarrow \Pi_{\infty}[\cI]$, which amalgamate to an injection (as $j^-(\widetilde{D}_1)$ and $j^+(\widetilde{C}_1)$ don't intertwine)
	\begin{equation*}
		j^-(\widetilde{D}_1) \oplus_{\pi_1(\ul{\alpha}, \lambda)} j^+(\widetilde{C}_1) \hooklongrightarrow \Pi_{\infty}[\cI].
	\end{equation*}
	Moreover, one can show that  the two copies of $\pi_{\alg}(\ul{\alpha}, \lambda)$ in the cosocle of $ 	j^-(\widetilde{D}_1) \oplus_{\pi_1(\ul{\alpha}, \lambda)} j^+(\widetilde{C}_1)$ (see (\ref{Eipic2}))  are both sent to the single  $\pi_{\alg}(\ul{\alpha}, \lambda)$ in the second  $\Pi_{\infty}[\fm]$ in (\ref{Ethick}). Consequently,  a certain subextension of $\pi_{\alg}(\ul{\alpha}, \lambda)$ by $\pi_1(\ul{\alpha},\lambda)$, of 	$j^-(\widetilde{D}_1) \oplus_{\pi_1(\ul{\alpha}, \lambda)} j^+(\widetilde{C}_1)$, has to inject into $\Pi_{\infty}[\fm]$. Letting $(\widetilde{D}_1, \widetilde{C}_1)$ vary, it is not so difficult to see the subextension has to be $\pi(\ul{\alpha}, \lambda, \iota_D)$. See the proof of Theorem \ref{Tmain2} for details.

	Remark that our arguments in fact do not rely on the full strength of the $p$-adic Langlands correspondence for $\GL_2(\Q_p)$ (for example, we don't use Colmez's functor). Most of the arguments can generalize to $\GL_n(\Q_p)$ (or even $\GL_n(K)$ for a finite extension $K$ of $\Q_p$). We will report the $\GL_n$-case in an upcoming work, which of course covers the results in the note. However, we find it convenient (for both the reader and the author) to first discuss the $\GL_3(\Q_p)$-case.
	
	We refer to the body of the context for more detailed and precise statements with slightly different notation.
	
	\subsection*{Acknowledgement}
	I thank Zicheng Qian, Zhixiang Wu for helpful discussions. I especially thank Christophe Breuil for the discussions and his interest, which  substantially accelerates the work, and for his comments on a preliminary version of the note. This work is supported by   the NSFC Grant No. 8200800065, No. 8200907289 and No. 8200908310.
	%
	
	
	

	\section{Hodge filtration and higher intertwining}\label{S1}
	Let $D$ be a $(\varphi, \Gamma)$-module of rank $3$ over $\cR_E$, the $E$-coefficient Robba ring for $\Q_p$. Assume $D$ is crystabelline  of regular  Hodge-Tate-Sen weights $h=(h_1> h_2> h_3)$.  We discuss the relation between the Hodge filtration of $D$ and certain paraboline deformations of $D$. Throughout the section, we write $\Ext^i$ (and $\Hom=\Ext^0$)  without ``$(\varphi, \Gamma)$" in the subscript for the $i$-th extension group of $(\varphi, \Gamma)$-modules (cf. \cite{Liu07}).

	\subsection{Hodge filtration}\label{S2.1}As $D$ is crystabelline, there exist smooth characters $\phi_1, \phi_2, \phi_3$ of $\Q_p^{\times}$ such that $D[\frac{1}{t}]\cong \oplus_{i=1}^3 \cR_E(\phi_i)[\frac{1}{t}]$. We assume $D$ is \textit{generic}, that is $\phi_i \phi_j^{-1}\neq 1, |\cdot |^{\pm 1}$ for $i \neq j$. An ordering of $\phi_1$, $\phi_2$, $\phi_3$ is refereed to as a \textit{refinement} of $D$. For $w\in S_3$, the refinement $(\phi_{w^{-1} (1)}, \phi_{w^{-1}(2)},\phi_{w^{-1}(3)})$ is called \textit{non-critical} if $D$ admits a triangulation of parameter $w(\ul{\phi})z^{h}$, where $w(\ul{\phi})=\phi_{w^{-1} (1)} \boxtimes \phi_{w^{-1}(2)} \boxtimes \phi_{w^{-1}(3)}$ is the associated smooth character of $T(\Q_p)$. We recall it means $D$ is isomorphic to a successive extension of $\cR_{E}(\phi_{w^{-1}(i)} z^{h_i})$ for $i=1, 2, 3$. We also call $w(\ul{\phi})$ a refinement of $D$.
	Throughout the note, we assume the following hypothesis
	\begin{hypothesis}\label{Hnonc}
		Assume all the refinements of $D$ are non-critical. 
	\end{hypothesis}
	Let $D_1$ (resp. $C_1$) be the submodule of $D$ (resp. the quotient of $D$) given by an extension of $\cR_E(z^{h_2} \phi_2)$ (resp. $\cR_E(z^{h_3} \phi_2)$) by $\cR_E(z^{h_1} \phi_1)$ (resp. by $\cR_E(z^{h_2} \phi_1)$). By the assumption, it is easy to see both $D_1$ and $C_1$ are non-split. 
	Denote by $\iota_D$ the composition $D_1 \hookrightarrow D \twoheadrightarrow C_1$, and it is clear that $\iota_D$ is injective (using $\Hom(D_1, \cR_E(z^{h_1} \phi_3))=0$). 
	\begin{proposition} \label{PHodge1}
		(1) $\dim_E \Hom(D_1, C_1)=2$.
		
		(2) The  cup-product
		\begin{equation*}
			\Ext^1(\cR_E(\phi_3 z^{h_3}),D_1) \times \Hom(D_1, C_1) \lra \Ext^1(\cR_E(\phi_3 z^{h_3}),C_1)
		\end{equation*}
		is non-degenerate. Moreover, $[D]$ is exactly the line orthogonal to $[\iota_D]$. In particular, $D$ is determined by $\ul{\phi}$, $h$ and $\iota_D$. 
	\end{proposition}
	\begin{proof}
		Consider $C_1 \otimes_{\cR_E} D_1^{\vee}$, which is isomorphic to a successive extension of $\cR_E(\phi_i\phi_j^{-1} z^{h_{i+1}-h_{j}})$ for $i, j\in \{1,2\}$. By d\'evissage, it is easy to see $\dim_E \Hom(D_1,C_1)\leq 2$. For $i=1,2$, let 
		\begin{equation}\label{Ealphai}
			\alpha_i:=D_1 \twoheadlongrightarrow \cR_E(z^{h_2} \phi_i) \hooklongrightarrow C_1
		\end{equation}
		which are obviously linearly independent. (1) follows. 
		
		We have a commutative diagram of cup-products
		\begin{equation*}
			\begin{CD}	
				\Ext^1(\cR_E(\phi_3 z^{h_3}),\cR_E(\phi_1 z^{h_2})) @. \ \  \times @. \  \ \Hom(\cR_E(\phi_1 z^{h_2}), \cR_E(\phi_1 z^{h_2})) @>>> \Ext^1(\cR_E(\phi_3 z^{h_3}),\cR_E(\phi_1 z^{h_2})) \\
				@AAA @. @V \sim VV @| \\
				\Ext^1(\cR_E(\phi_3 z^{h_3}),D_1) @. \ \  \times @. \  \ \Hom(D_1, \cR_E(\phi_1 z^{h_2})) @>>> \Ext^1(\cR_E(\phi_3 z^{h_3}),\cR_E(\phi_1 z^{h_2})) \\
				@| @. @VVV @VVV \\
				\Ext^1(\cR_E(\phi_3 z^{h_3}),D_1) @. \ \  \times @. \  \ \Hom(D_1, C_1) @>>> \Ext^1(\cR_E(\phi_3 z^{h_3}),C_1).
			\end{CD}
		\end{equation*}
		The top pairing (of one dimensional spaces) is trivially perfect. We deduce 
		$$[\alpha_1]^{\perp}=\Ext^1(\cR_E(\phi_3 z^{h_3}), \cR_E(\phi_2 z^{h_1})) (\subset \Ext^1(\cR_E(\phi_3 z^{h_3}),D_1)).$$ Similarly, we have $[\alpha_2]^{\perp}=\Ext^1(\cR_E(\phi_3 z^{h_3}), \cR_E(\phi_1 z^{h_1}))$. In particular, $[\alpha_1]^{\perp}$ and $[\alpha_2]^{\perp}$ are linearly disjoint. The first part of (2) follows.  As $\iota_D$ factors through $D$, the map induced by the pairing $\langle -, \iota_D\rangle$ is given by the following  composition
		\begin{equation*}
			\Ext^1(\cR_E(\phi_3 z^{h_3}), D_1) \lra 	\Ext^1(\cR_E(\phi_3 z^{h_3}), D) \lra 	\Ext^1(\cR_E(\phi_3 z^{h_3}), C_1).\end{equation*}
		The first map sends  $[D]$ to zero, hence $\langle D, \iota_D\rangle=0$. This finishes the proof. 
	\end{proof}
	\begin{remark}\label{Rfphi}
		By the proposition, the map $\iota_D$ determines the Hodge filtration of $D$. We can actually see this in an explicit way. For simplicity, we assume $D$ is crystalline. In this case, $\phi_i=\unr(\alpha_i)$ where $\alpha_i$ are the eigenvalues of $\varphi$ on $D_{\cris}(D)$. Let $e_i\in D_{\cris}(D)$ be an $\alpha_i$-eigenvector. As $D_1$ is non-split, by multiplying $e_1$, $e_2$ by non-zero scalars, we can and do assume $\Fil^{\max} D_{\cris}(D_1)=\Fil^j D_{\cris}(D_1)$, $-h_1<j\leq -h_2$, is generated by $e_1+e_2$. As $D$ is non-critical for all refinements, by multiplying $e_3$ by a non-zero scalar, we can and do assume $\Fil^{\max} D_{\cris}(D)=\Fil^j D_{\cris}(D)$, $-h_2 < j \leq -h_3$, is generated by $e_1+a_D e_2+e_3$ for some $a_D \in E^{\times}$. The Hodge filtration on $D_{\cris}(D)$ is in fact determined by the parameter $a_D$:
		\begin{equation*}
			\Fil^j D_{\cris}(D)=\begin{cases}
				D_{\cris}(D) & j\leq -h_1 \\
				E(e_1+e_2) \oplus E(e_1+a_D e_2+e_3) & -h_1<j \leq -h_2 \\
				E(e_1+a_D e_2+e_3) & -h_2<j \leq -h_3 \\
				0 & j>-h_3
			\end{cases}
		\end{equation*}
		With these basis, $\Fil^{\max} D_{\cris}(C_1)=\Fil^j D_{\cris}(C_1)$, $-h_2<j \leq -h_3$, is generated by $e_1+a_D e_2$. The map $\iota_D$ induces a map $D_{\cris}(D_1) \ra D_{\cris}(C_1)$ which sends exactly $e_i$ to $e_i$. We see $a_D$ can be read out from the relative position of the two lines $\iota_D(\Fil^{\max} D_{\cris}(D_1))$ and $\Fil^{\max} D_{\cris}(C_1)$. It is then not difficult to see $a_D$ and $\iota_D$ determine each other. 
	\end{remark}
	
	\subsection{Higher intertwining pairs}
	We show that $\iota_D$ in $\Hom(D_1,C_1)$ can be detected by a set of pairs of certain deformations of $D_1$ and $C_1$, associated to $D$.

	For de Rham $(\varphi, \Gamma)$-modules $M$ and $N$, let  $\Ext^1_g(M,N)$ denote the subspace of de Rham extensions.  We identify each element in $\Ext^1(M,M)$ with a $(\varphi, \Gamma)$-module $\widetilde{M}$ over $\cR_{E[\epsilon]/\epsilon^2}$ where $\epsilon$ acts via $\widetilde{M}\twoheadrightarrow M \xrightarrow{\id} M \hookrightarrow M$. Let $\Ext^1_{g'}(M,M)$ be the subspace of de Rham deformations up to twist by characters (over $E[\epsilon]/\epsilon^2$).

	We quickly recall some facts on the deformations of $D_1$ (similar statements holding for $C_1$ as well).
	Let $\sF_1$ (resp. $\sF_2$) be the filtration $\cR_{E}(\phi_1 z^{h_1}) \subset D_1$ (resp. $\cR_E(\phi_2 z^{h_1}) \subset D_1$), and $\Ext^1_{\sF_i}(D_1,D_1)$ be subspace of trianguline  deformations with respect to the filtration $\sF_i$.  For $\widetilde{D}_1\in \Ext^1_{\sF_i}(D_1,D_1)$ (viewed as a $(\varphi, \Gamma)$-module over $\cR_{E[\epsilon]/\epsilon^2}$), there exist characters $\widetilde{\delta}_{\sF_i,1}, \widetilde{\delta}_{\sF_i,2}: \Q_p^{\times} \ra (E[\epsilon]/\epsilon^2)^{\times}$, $\widetilde{\delta}_{\sF_i,1}\equiv \delta_{\sF_i,1}:=\phi_i z^{h_1} \pmod{\epsilon}$, and $\widetilde{\delta}_{\sF_i,2}\equiv \delta_{\sF_i,2}:=\phi_j z^{h_2}$ (for $j \neq i$) such that $\widetilde{D}_1$ is isomorphic, as $(\varphi, \Gamma)$-module over $\cR_{E[\epsilon]/\epsilon^2}$, to an extension of $\cR_{E[\epsilon]/\epsilon^2}(\widetilde{\delta}_{\sF_i,2})$ by $\cR_{E[\epsilon]/\epsilon^2}(\widetilde{\delta}_{\sF_i,1})$.  We call  the character  $\widetilde{\delta}_{\sF_i}:=\widetilde{\delta}_{\sF_i,1} \boxtimes \widetilde{\delta}_{\sF_i,2}$ of $T(\Q_p)$ a trianguline parameter of  $\widetilde{D}_1$.  The following proposition is well-known.

	\begin{proposition}\label{PGL21}(1) $\dim_E \Ext^1(D_1,D_1)=5$,  $\dim_E \Ext^1_g(D_1,D_1)=2$ and $\dim_E \Ext^1_{\sF_i}(D_1,D_1)=4$ for $i=1,2$. 
		
		(2) There is a natural exact sequence
		\begin{equation*}
			0 \ra \Ext^1_{g'}(D_1,D_1) \ra \Ext^1_{\sF_1}(D_1,D_1) \oplus \Ext^1_{\sF_2}(D_1,D_1) \ra \Ext^1(D_1,D_1) \ra 0.
		\end{equation*}
		
		(3) For $i=1,2$, the map 
		$\Ext^1_{\sF_i}(D_1,D_1) \ra\Ext^1_{T(\Q_p)}(\delta_{\sF_i}, \delta_{\sF_i})$, sending $\widetilde{D}_1$ to its trianguline parameter, is a bijection, where $\delta_{\sF_i}:=\delta_{\sF_i,1}\boxtimes \delta_{\sF_i,2}$.
	%
	%
	\end{proposition}
\begin{remark}\label{Rint01}
	For $\widetilde{D}_1\in \Ext^1_{g'}(D_1,D_1)$, there exist a continuous character $\psi \in \Hom(\Q_p^{\times}, E)$ and smooth characters $\psi_i\in \Hom_{\sm}(\Q_p^{\times},E)$ such that both $\phi_1 z^{h_1}(1+(\psi_1+\psi/2)\epsilon) \boxtimes \phi_2 z^{h_2}(1+(\psi_2 +\psi/2)\epsilon)$ and $\phi_2 z^{h_1}(1+(\psi_2+\psi/2)\epsilon) \boxtimes \phi_1 z^{h_2}(1+(\psi_1 +\psi/2)\epsilon)$ are trianguline parameters of $\widetilde{D}_1$. We denote by 
	\begin{equation}\label{Ekappa}\kappa: \Ext^1_{g'}(D_1,D_1) \lra \Hom(\Q_p^{\times},E) \times \Hom_{\sm}(\Q_p^{\times},E)^{\oplus 2}
		\end{equation}the map sending $\widetilde{D}_1$ to $(\psi, \psi_1, \psi_2)$. Remark that the map is compatible with the map in (3). 
\end{remark}
	
Now let $\sF$ be the filtration $D_1 \subset D$. A deformation $\widetilde{D}$ of $D$ over $\cR_{E[\epsilon]/\epsilon^2}$ is called an \textit{$\sF$-deformation}, if there exist a deformation $\widetilde{D}_1$ of $D_1$ over $\cR_{E[\epsilon]/\epsilon^2}$ and a deformation  $\cR_{E[\epsilon]/\epsilon^2}(\phi_3z^{h_3}(1+\psi \epsilon))$ of $\cR_E(\phi_3z^{h_3})$ over $\cR_{E[\epsilon]/\epsilon^2}$ (where $\psi\in \Hom(\Q_p^{\times},E)$), such that $\widetilde{D}$ is isomorphic to an extension of $\cR_{E[\epsilon]/\epsilon^2}(\phi_3z^{h_3}(1+\psi \epsilon))$ by $\widetilde{D}_1$. Denote by $\Ext^1_{\sF}(D,D)\subset \Ext^1(D,D)$ the subspace of $\sF$-deformations.
	\begin{lemma}\label{Lpar1}
		The subspace $\Ext^1_{\sF}(D,D)$ is the kernel of the following composition of surjective maps
		\begin{equation}\label{Epar1}
			\Ext^1(D,D) \twoheadlongrightarrow \Ext^1(D_1, D) \twoheadlongrightarrow\Ext^1(D_1, \cR_E(\phi_3 z^{h_3})).
		\end{equation}
		In particular, $\dim_E\Ext^1_{\sF}(D,D)=8$. 
	\end{lemma}
	\begin{proof}
		The surjectivity follows easily from d\'evissage and the fact $D$ is generic. The first part follows by definition. We have $\dim_E \Ext^1(D,D)=10$ and $\dim_E \Ext^1(D_1, \cR_E(\phi_3 z^{h_3}))=2$. The second part follows.
	\end{proof}
	\begin{proposition}\label{Ppar1}
		There is a natural exact sequence ($\cR$ for $\cR_E$)
		\begin{equation*}
			0 \ra \Ext^1(\cR(\phi_3 z^{h_3}),D_1) / E ([D]) \xrightarrow{j} \Ext^1_{\sF}(D,D) \ra  \Ext^1(D_1, D_1) \times \Ext^1(\cR(\psi_3z^{h_3}), \cR(\psi_3 z^{h_3})) \ra 0.
		\end{equation*}
	\end{proposition}
	\begin{proof}
		By Lemma \ref{Lpar1} and d\'evissage, there is an exact sequence
		\begin{multline*}
			0 \lra \Hom(\cR_E(\phi_3 z^{h_3}), \cR_E(\phi_3 z^{h_3})) 
			\lra \Ext^1(\cR_E(\phi_3 z^{h_3}), D_1) \\ \lra \Ext^1(\cR_E(\phi_3 z^{h_3}),D) \lra \Ext^1_{\sF}(D,D) \lra \Ext^1(D_1,D_1) \ra 0.
		\end{multline*}
		The image of the first map is $E([D])$,  and the cokernel of the second map is $\Ext^1(\cR_E(\phi_3 z^{h_3}), \cR_E(\phi_3 z^{h_3}))$. Letting $V$ be the cokernel of $j$, we have an exact sequence
		\begin{equation*}
			0 \lra \Ext^1(\cR_E(\phi_3 z^{h_3}), \cR_E(\phi_3 z^{h_3})) \lra V \lra \Ext^1(D_1, D_1) \lra 0.
		\end{equation*}Meanwhile, the composition (\ref{Epar1}) factors through
		\begin{equation*}
			\Ext^1(D,D) \twoheadlongrightarrow \Ext^1(D, \cR_E(\phi_3 z^{h_3}) \twoheadlongrightarrow\Ext^1(D_1, \cR_E(\phi_3 z^{h_3})).
		\end{equation*}
		We deduce an exact sequence
		\begin{multline}\label{Eexa1}
			0 \lra \Hom(D_1, D_1) 
			\lra \Ext^1(\cR_E(\phi_3 z^{h_3}), D_1) \\ \lra \Ext^1(D,D_1) \lra \Ext^1_{\sF}(D,D) \lra \Ext^1(\cR_E(\phi_3 z^{h_3}), \cR_E(\phi_3 z^{h_3})) \lra 0.
		\end{multline}
		Hence $V$ is actually split. The proposition follows.
	\end{proof}
	\begin{remark}\label{Rkappf}(1) We identify $\Ext^1(\cR_E(\phi_3z^{h_3}), \cR_E(\phi_3 z^{h_3}))$ with $\Hom(\Q_p^{\times},E)$, and denote by $\kappa_{\sF}$ the induced surjection
		\begin{equation*}
			\kappa_{\sF}=(\kappa_{\sF,1}, \kappa_{\sF,2}): \Ext^1_{\sF}(D,D) \twoheadlongrightarrow  \Ext^1(D_1, D_1) \times \Ext^1(\cR_E(\phi_3z^{h_3}), \cR_E(\phi_3 z^{h_3})). 
		\end{equation*}For $\widetilde{D}\in \Ext^1_{\sF}(D,D)$ of the form $[\widetilde{D}_1 \lin \cR_{E[\epsilon]/\epsilon^2}(\phi_3z^{h_3} (1+\psi \epsilon))]$, $\kappa_{\sF}$ sends $\widetilde{D}$ to $(\widetilde{D}_1, \psi)$.
		
		(2) Let $\sG$ be the filtration $\cR_E(\phi_3 z^{h_1}) \subset D$, and fix $\cR_E(\phi_3 z^{h_1}) \hookrightarrow D \twoheadrightarrow C_1$. We define $\sG$-deformations of $D$ in a similar way, and denote by $\Ext^1_{\sG}(D,D)$ the subspace of $\sG$-deformations, which is of dimension $8$. Similarly, we have an exact sequence
		\begin{equation*}
			0 \ra \Ext^1(C_1,\cR_E(\phi_3 z^{h_1})) / E ([D]) \ra \Ext^1_{\sG}(D,D) \xrightarrow{\kappa_{\sG}}  \Ext^1(C_1, C_1) \times \Hom(\Q_p^{\times},E) \ra 0.
		\end{equation*}
	\end{remark}
	For $\iota\in \Hom(D_1,C_1)$. Consider the pull-back and push-forward maps:
	\begin{equation*}
		\iota^-: \Ext^1(C_1, D_1) \lra \Ext^1(D_1, D_1), \ \iota^+: \Ext^1(C_1, D_1) \lra \Ext^1(C_1, C_1).
	\end{equation*}
	Put
	\begin{equation*}
		\Ext^1_{\iota}(D_1,D_1):=\iota^-(\Ext^1(C_1, D_1)), \ \Ext^1_{\iota}(C_1, C_1):=\iota^+(\Ext^1(C_1,D_1)).
	\end{equation*}
	\begin{lemma}\label{Lalphai}
		For $i\in \{1,2\}$, we have $\dim_E \Ext^1_{\alpha_i}(D_1, D_1) =2$. Moreover,
		\begin{equation*}
		\Ext^1_{\alpha_1}(D_1, D_1) \cap \Ext^1_{\alpha_2}(D_1,D_1)=0. 
		\end{equation*}
		The same holds with $D_1$ replaced by $C_1$.
	\end{lemma}
	\begin{proof}
		We only prove it for $D_1$, with $C_1$ being similar. The map $\alpha_i^-$ factors through
		\begin{equation*}
			\Ext^1(C_1, D_1) \twoheadlongrightarrow \Ext^1(\cR_E(\phi_i z^{h_2}), D_1) \hooklongrightarrow \Ext^1(D_1, D_1)
		\end{equation*}
		where the corresponding  surjectivity and injectivity follow easily by d\'evissage. The first part then follows from the fact
		\begin{equation*}
			\dim_E \Ext^1(\cR_E(\phi_i z^{h_2}), D_1)=2.
		\end{equation*}
		We also see $\Ext^1_{\alpha_i}(D_1,D_1)$ is just the kernel of $\Ext^1(D_1, D_1) \ra \Ext^1(\cR_E(\phi_j z^{h_1}), D_1)$ for $j\neq i$. We then easily deduce $\Ext^1_{\alpha_1}(D_1, D_1) \cap \Ext^1_{\alpha_2}(D_1,D_1)=0$. 
	\end{proof}
	\begin{proposition}\label{Ppairing}
		Let $\iota\in \Hom(D_1, C_1)$ be an injection.
		
		(1)  $\dim_E \Ext^1_{\iota}(D_1, D_1)=\dim_E \Ext^1_{\iota}(C_1,C_1)=3$.
		
		(2) $\Ext_g^1(D_1, D_1) \subset \Ext^1_{\iota}(D_1,D_1)$ and $\Ext_g^1(C_1, C_1) \subset \Ext^1_{\iota}(C_1,C_1)$. Moreover, any trianguline deformation in $\Ext^1_{\iota}(D_1,D_1)$ (resp. in $\Ext^1_{\iota}(C_1,C_1)$) is de Rham.
		
		(3) For $\iota'\in \Hom(D_1,C_1)$, $\Ext^1_{\iota'}(D_1, D_1)=\Ext^1_{\iota}(D_1,D_1)$ if and only if  $\Ext^1_{\iota'}(C_1, C_1)=\Ext^1_{\iota}(C_1,C_1)$ if and only if $\iota'=a \iota$ for some $a\in E^{\times}$.
	\end{proposition}
	\begin{proof}
		We only prove it for $D_1$ with $C_1$ being similar. Note first $\dim_E\Ext^1(C_1, D_1)=4$. Consider the commutative diagram
		\begin{equation}\label{Edivi1}
			\begin{CD}
				0 @>>> \Ext^1(C_1, \cR_E(\phi_1 z^{h_1})) @>>> \Ext^1(C_1, D_1) @>>> \Ext^1(C_1, \cR_E(\phi_2 z^{h_2})) @>>> 0 \\ @. @VVV @V \iota^- VV @VVV @. \\
				0 @>>> \Ext^1(D_1, \cR_E(\phi_1 z^{h_1})) @>>> \Ext^1(D_1, D_1) @>>> \Ext^1(D_1,  \cR_E(\phi_2 z^{h_2}))@>>> 0.
			\end{CD}
		\end{equation}
By d\'evissage and 	using \cite[Lem.~5.1.1]{BD2}, it is not difficult to see  the right vertical map is injective.  We can furthermore d\'evissage the left vertical map of (\ref{Edivi1}):
		\begin{equation*}
			\begin{CD}
				\Ext^1(\cR_E(\phi_2 z^{h_3}), \cR_E(\phi_1 z^{h_1})) @>>> \Ext^1(C_1, \cR_E(\phi_1 z^{h_1})) @>>> \Ext^1(\cR_E(\phi_1 z^{h_2}), \cR_E(\phi_1 z^{h_1}))\\
				@VVV @VVV @VVV \\
				\Ext^1(\cR_E(\phi_2 z^{h_2}), \cR_E(\phi_1 z^{h_1})) @> 0 >> \Ext^1(D_1, \cR_E(\phi_1 z^{h_1})) @>>> \Ext^1(\cR_E(\phi_1 z^{h_1}), \cR_E(\phi_1 z^{h_1})).
			\end{CD}
		\end{equation*}
		As $h_2<h_1$, the image of the right vertical map is exactly $\Ext^1_g(\cR_E(\phi_1 z^{h_1}), \cR_E(\phi_1 z^{h_1}))$, which is one  dimensional.  Together with $\dim_E \Ext^1(C_1, \cR_E(\phi_2 z^{h_2}))=2$, (1) follows.

		For a $(\varphi, \Gamma)$-module $M$ over $\cR_E$, denote by $W_{\dR}^+(M)$ the associated $B_{\dR,E}^+$-representation of $\Gal_{\Q_p}$, where $B_{\dR,E}^+:=B_{\dR}^+ \otimes_{\Q_p} E$. We have a tautological  exact sequence
		\begin{equation*}
			H^1_g(D_1 \otimes_{\cR_E} C_1^{\vee}) \hooklongrightarrow H^1(D_1 \otimes_{\cR_E} C_1^{\vee}) \lra H^1(\Gal_K, W_{\dR}^+(D_1 \otimes_{\cR_E} C_1^{\vee})).
		\end{equation*}
		It is not difficult to see $\dim_E H^1(\Gal_K, W_{\dR}^+(D_1 \otimes_{\cR_E} C_1^{\vee}))=1$, hence $\dim_E\Ext^1_g(C_1,D_1)=\dim_E H^1_g(D_1 \otimes_{\cR_E} C_1^{\vee}) \geq 3$. As $\iota^-$ obviously induces
		\begin{equation*}
			\iota_g^-:	\Ext^1_g(C_1, D_1) \lra \Ext^1_g(D_1,D_1),
		\end{equation*}
		by comparing the dimensions, we see $\Ker \iota_g^-=\Ker \iota^-$ is one dimensional and $\iota_g^-$ is surjective (and $\dim_E \Ext^1_g(C_1,D_1)=3$). 
		
		We have seen $\Ext^1_g(D_1,D_1)\subset \Ext^1_{\iota}(D_1,D_1) \cap \Ext^1_{\sF_1}(D_1,D_1)$. If it is not an equality, by comparing the dimension, $\Ext^1_{\iota}(D_1,D_1) \cap \Ext^1_{\sF_1}(D_1,D_1)=\Ext^1_{\iota}(D_1,D_1)$. Let $\widetilde{D}_1\in \Ext^1_{\iota}(D_1,D_1) \cap \Ext^1_{\sF_1}(D_1,D_1)$. By (\ref{Edivi1}), the image of $\widetilde{D}_1$ in $\Ext^1(D_1, \cR_E(\phi_2 z^{h_2}))$ lies in $\Ext^1(C_1, \cR_E(\phi_2z^{h_2}))\cap \Ext^1(\cR_E(\phi_2z^{h_2}), \cR_E(\phi_2 z^{h_2}))$, which, by \cite[Lem.~5.5.9]{BD2}, is isomorphic to the one dimensional $\Ext^1_g(\cR_E(\phi_2 z^{h_2}), \cR_E(\phi_2 z^{h_2}))$ hence is not equal to the whole $\Ext^1(C_1, \cR_E(\phi_2 z^{h_2}))$. Thus we have  $\Ext^1_{\iota}(D_1,D_1)\cap \Ext^1_{\sF_1}(D_1,D_1) \neq \Ext^1_{\iota}(D_1,D_1)$.  The same  holds with $\sF_1$ replaced by $\sF_2$, and  (2) follows. 
		
		 Finally,  consider the cup-product
		\begin{equation*}
			\Ext^1(C_1,D_1) \times \Hom(D_1,C_1) \lra \Ext^1(D_1,D_1).
		\end{equation*}
		Suppose $\iota'\notin E[\iota]$, then $\iota'$ and $\iota$ form a basis of $\Hom(D_1, C_1)$. If $\Ext^1_{\iota'}(D_1, D_1)=\Ext^1_{\iota}(D_1,D_1)$, we then easily deduce $\Ext^1_{\alpha_i}(D_1, D_1) \subset \Ext^1_{\iota}(D_1, D_1)$ for all $i=\{1,2\}$. However, by Lemma \ref{Lalphai}, $\dim_E \big(\Ext^1_{\alpha_1}(D_1, D_1)+\Ext^1_{\alpha_2}(D_1, D_1)\big)=4$, contradiction.
	\end{proof}
Denote by $\kappa: \Ext^1_g(D_1,D_1)\ra \Hom_{\sm}(\Q_p^{\times},E)^{\oplus 2}$ the (bijective) map induced by (\ref{Ekappa}). And we have a similar bijection $\kappa: \Ext^1_g(C_1, C_1) \twoheadrightarrow \Hom_{\sm}(\Q_p^{\times}, E)^2$.
	\begin{proposition}\label{Pamal1}
		For $M\in \Ext^1_g(C_1, D_1)$, $\kappa \circ \iota_g^- (M)=\kappa \circ \iota_g^+(M)$.
	\end{proposition}
	\begin{proof}
		By definition, there is a natural injection $\tilde{\iota}: \iota_g^-(M) \hookrightarrow \iota_g^+(M)$ which sits in the following commutative diagram
		\begin{equation*}
			\begin{CD}
				0 @>>> D_1 @>>> \iota_g^-(M) @>>> D_1 @>>> 0 \\
				@. @V \iota VV @V \tilde{\iota} VV @V \iota VV \\
				0 @>>> C_1 @>>> \iota_g^+(M) @>>> C_1 @>>> 0.
			\end{CD}
		\end{equation*}
		It is easy to see $\tilde{\iota}$ is moreover $\cR_{E[\epsilon]/\epsilon^2}$-linear if $\iota_g^-(M)$ and $\iota_g^+(M)$ are equipped with the natural $\cR_{E[\epsilon]/\epsilon^2}$-action. Suppose $\kappa \circ \iota_g^-(M)=(\psi_1,\psi_2)$ and $\kappa \circ \iota_g^+(M)=(\psi_1',\psi_2')$. Then $\iota_g^-(M)$ (resp. $\iota_g^+(M)$) is isomorphic, as $(\varphi, \Gamma)$-module over $\cR_{E[\epsilon]/\epsilon^2}$, to an extension of $\cR_{E[\epsilon]/\epsilon^2}(\phi_2 z^{h_2} (1+\psi_2 \epsilon))$ \big(resp. of $\cR_{E[\epsilon]/\epsilon^2}(\phi_2 z^{h_{3}}(1+\psi'_2 \epsilon))$\big)  by $\cR_{E[\epsilon]/\epsilon^2}(\phi_1 z^{h_1}(1+\psi_1\epsilon))$ \big(resp. $\cR_{E[\epsilon]/\epsilon^2}(\phi_1 z^{h_2}(1+\psi'_1 \epsilon))$\big). It is not difficult to see $\tilde{\iota}$ induces injections $\cR_{E[\epsilon]/\epsilon^2}(\phi_i z^{h_i} (1+\psi_i \epsilon)) \hookrightarrow \cR_{E[\epsilon]/\epsilon^2}(\phi_i z^{h_{i+1}} (1+\psi_i' \epsilon))$ of $(\varphi, \Gamma)$-modules over $\cR_{E[\epsilon]/\epsilon^2}$. Hence $\psi_i=\psi_i'$ for $i=1,2$.
	\end{proof}
	For an injection $\iota\in \Hom(D_1,C_1)$, we define $\sI_{\iota}$ to be following set:
	\begin{equation*}
		\{(\widetilde{D}_1, \widetilde{C}_1)\in \Ext^1_{\iota}(D_1,D_1) \times \Ext^1_{\iota}(C_1,C_1)\ |\ \exists M\in \Ext^1(C_1, D_1) \text{ with } \iota^-(M)=\widetilde{D}_1, \iota^+(M)=\widetilde{C}_1\}. 
	\end{equation*}
	If $\iota=\iota_D$ for some $D$ as in \S~\ref{S2.1}, we write $\sI_D:=\sI_{\iota_D}$. By Proposition \ref{Ppairing} (3) and Proposition \ref{PHodge1} (2), we have:
	\begin{corollary}
		We have $\sI_{\iota}=\sI_{\iota'}$ if and only if $\iota'=a \iota$ for some $a\in E^{\times}$. In particular, for non critical $D$ and $D'$, which both have  Sen weights $(h_1, h_2,h_3)$ and a refinement $\phi$, then $\sI_D=\sI_{D'}$ if and only if $D\cong D'$.
	\end{corollary}
	Now let $D$ be as in \S~\ref{S2.1}, and $\iota_D: D_1 \hookrightarrow C_1$ be the associated injection. 
	\begin{theorem}[Higher intertwining]\label{ThIW1}Let $\widetilde{D}\in \Ext^1_{\sF}(D, D)$ with $\kappa_{\sF}(\widetilde{D})=(\widetilde{D}_1, \psi)$. The followings are equivalent:
		\begin{enumerate}
			\item  $\widetilde{D} \in \Ext^1_{\sF}(D,D) \cap \Ext^1_{\sG}(D,D)$.
			\item  $\widetilde{D}_1 \otimes_{\cR_{E[\epsilon]/\epsilon^2}} \cR_{E[\epsilon]/\epsilon^2}(1-\psi \epsilon)\in \Ext^1_{\iota_D}(D_1,D_1)$.
		\end{enumerate}		
		Moreover, if the equivalent conditions hold, then $\kappa_{\sG,2}(\widetilde{D})=\psi$ and	 there exists $M\in \Ext^1(C_1,D_1)$ such that $\widetilde{D}_1=\iota_D^-(M) \otimes_{\cR_{E[\epsilon]/\epsilon^2}} \cR_{E[\epsilon]/\epsilon^2}(1+\psi \epsilon)$ and $\kappa_{\sG,1}(\widetilde{D})=\iota_D^+(M)\otimes_{\cR_{E[\epsilon]/\epsilon^2}}\cR_{E[\epsilon]/\epsilon^2}(1+\psi \epsilon)$.
	\end{theorem}
	\begin{proof}
		Twisting by $1-\psi \epsilon$, we assume $\kappa_{\sF,2}(\widetilde{D})=0$. Using a similar statement in Lemma \ref{Lpar1} (for $\sG$), $\widetilde{D}\in \Ext^1_{\sG}(D,D)$  if and only if it lies in the kernel of the composition
		\begin{equation}\label{0}
			\Ext^1(D,D) \lra \Ext^1(\cR_E( \phi_3 z^{h_1}), D) \lra \Ext^1(\cR_E(\phi_3 z^{h_1}), C_1).
		\end{equation}
		As $\kappa_{\sF,2}(\widetilde{D})=0$, $\widetilde{D}$ lies in the image of $\Ext^1(D,D_1) \ra \Ext^1_{\sF}(D,D)$ (see (\ref{Eexa1})), and we let $M_1\in \Ext^1(D,D_1) $ be a preimage of $\widetilde{D}$. Consider the composition 
		\begin{equation*}
			\Ext^1(D,D_1) \hooklongrightarrow	\Ext^1(D,D) \lra \Ext^1(\cR_E(\phi_3 z^{h_1}), D) \lra \Ext^1(\cR_E( \phi_3 z^{h_1}), C_1).
		\end{equation*}
		It is straightforward to see it is equal to the composition
		\begin{equation}\label{EDtoM}
			\Ext^1(D,D_1) \lra \Ext^1(\cR_E(\phi_3 z^{h_1}), D_1) \xlongrightarrow{\iota_D} \Ext^1(\cR_E(\phi_3 z^{h_1}), C_1).
		\end{equation}
		So $\widetilde{D}$ lies in the kernel of (\ref{0}) if and only if $M_1$ is sent to zero via (\ref{EDtoM}). However, using d\'evissage and \cite[Lem.~5.1.1]{BD2}, the push-forward map $\iota_D$ in (\ref{EDtoM}) is injective. We see (under the assumption $\psi=0$) the condition 1 is equivalent to that   $M_1$ lies in the kernel of the first map of (\ref{EDtoM}), that is equal to $\Ext^1(C_1, D_1)$ by d\'evissage. This is furthermore equivalent to that  $\widetilde{D}_1$ lies in the image of the composition 
		\begin{equation}\label{EiotaD1}
			\Ext^1(C_1,D_1) \hooklongrightarrow \Ext^1(D,D_1) \lra \Ext^1(D_1,D_1),
		\end{equation}
		which is no other than the pull-back map induced by $\iota_D$.  The other parts are straightforward. 
	\end{proof}
	\section{Locally analytic representations of $\GL_3(\Q_p)$}
	Let $D$ be as in  \S~\ref{S2.1}, and let $\lambda=(\lambda_1, \lambda_2, \lambda_3):=(h_1-2,h_2-1,h_3)$. We associate to $D$ a locally analytic representation $\pi(\ul{\phi}, \lambda, \iota_D)$ of $\GL_3(\Q_p)$. We show that the construction gives a one-to-one correspondence between $\{\pi(\ul{\phi}, \lambda, \iota_D)\}$ and $\{D\}$.
	
	\subsection{Preliminaries and notation}\label{S3.1}
	For a smooth character $\delta=\delta_1 \boxtimes \delta_2 \boxtimes \delta_3: T(\Q_p) \ra E^{\times}$, denote by $\jmath(\delta):=\delta_1|\cdot|^{-2} \boxtimes \delta_2 |\cdot|^{-1} \boxtimes \delta_3$, where $|\cdot|$ is the $p$-adic norm with $|p|=p^{-1}$. Let $\varepsilon:=z |\cdot|: \Q_p^{\times} \ra E^{\times}$ be the cyclotomic character.   
	
	Let $T$ be the torus subgroup of $\GL_3$ , $B\supset T$ be the Borel subgroup of upper triangular matrices. For a standard parabolic subgroup $P$ of  $\GL_3$ (containing $B$), denote by $P^-$ its opposite parabolic subgroup.  For a weight $\mu$ of $\gl_3$, denote by $M^-(\mu):=\text{U}(\gl_3)  \otimes_{\text{U}(\ub^-)} \mu$ (with $\ub^-$ the Lie algebra of $B^-$), and let $L^-(\mu)$ be its unique simple quotient. If $\mu$ is anti-dominant, then $L^-(\mu)$ is finite dimensional isomorphic to $L(-\mu)^{\vee}$, where $L(-\mu)$ is the algebraic representation of highest weight $-\mu$ with respect to $B$. We use the same notation for $\GL_2$ when there is no ambiguity.
	
	Let $D$ be as in the Section \ref{S2.1}. For a refinement $w(\ul{\phi})$ of $D$, consider the locally algebraic representation $(\Ind_{B^-}^{\GL_3} \jmath(w(\ul{\phi})))^{\infty} \otimes_EL(\lambda)$, which turns out to be all isomorphic, that we denote by $\pi_{\alg}(\ul{\phi}, \lambda)$. In fact, the smooth induction $(\Ind_{B^-}^{\GL_3} \jmath(\ul{\phi}))^{\infty}$ is just the representation corresponding to the Weil-Deligne representation $\oplus_{i=1}^3 \phi_3$ via the classical local Langlands correspondence.

	For $w\in S_3$ and a simple reflection $s_i$, put (where ``$\cF$" is the notation for Orlik-Strauch representations as in \cite{OS})
	\begin{equation*}
		\sC(s_i,w):=\cF_{B^-}^{\GL_3}\big(L^-(-s_i \cdot \lambda), \jmath(w(\ul{\phi}))\big) \cong \cF_{P_j^-}^{\GL_3}\big(L^-(-s_i \cdot \lambda), (\Ind_{B^- \cap L_j}^{L_j} \jmath(w(\ul{\phi})))^{\infty}\big),
	\end{equation*}
	where $j\neq i$, $P_1^-=\begin{pmatrix}
		\GL_2 & 0 \\  
		* & \GL_1
	\end{pmatrix}$ and $P_2^-=\begin{pmatrix}
		\GL_1 & 0 \\  
		* & \GL_2
	\end{pmatrix}$.
	By \cite[Thm.]{OS}, $\sC(s_i,w)\cong \sC(s_k, w')$ if and only if $s_i=s_k$ and $w'=s_jw$ with $s_j \neq s_i$.   Letting $s_i$ and $w$ vary and assuming $w$ has minimal length in the class $\{w, s_jw\}_{s_j\neq s_i}$, there are exactly  $6$ distinct representations $\cS:=\{\sC(s_1, 1), \sC(s_1,s_1), \sC(s_1, s_1s_2), \sC(s_2,1), \sC(s_2,s_2), \sC(s_2, s_1s_2)\}$. 
	
	For $w\in S_3$, consider the principal series
	\begin{equation*}
		(\Ind_{B^-}^{\GL_3} \jmath(w(\ul{\phi})) z^{\lambda})^{\an}\cong \cF_{B^-}^{\GL_3}\big(M^-(-\lambda), \jmath(w(\ul{\phi}))\big).
	\end{equation*}
	It contains a unique subrepresentation $\pi_1(\ul{\phi},\lambda, w)$ of the form $[\pi_{\alg}(\ul{\phi},\lambda) \lin (\sC(s_1,w) \oplus \sC(s_2,w))]$, where each sub-extension $[\pi_{\alg}(\ul{\phi},\lambda) \lin \sC(s_i, w)]$ is non-split. Let $\pi_1(\ul{\phi}, \lambda)$ be the unique quotient of $\oplus_{w\in S_3} \pi_1(\ul{\phi}, \lambda, w)$ of socle $\pi_{\alg}(\ul{\phi},\lambda)$, which is an extension of $\oplus_{\sC\in \cS} \sC$ by $\pi_{\alg}(\ul{\phi},\lambda)$ with each sub-extension $[\pi_{\alg}(\ul{\phi}, \lambda) \lin \sC]$ non-split. 
	
	Now let $D_1\subset D$ be as in \S~\ref{S2.1}, and $\pi(D_1)$ be the locally analytic representation of $\GL_2(\Q_p)$ associated to $D_1$. We recall the structure of $\pi(D_1)$. Let $\lambda^1:=(h_1-1, h_2)$. For $i,j\in \{1,2\}$, $i\neq j$, let  $I_{\sF_i}:=(\Ind_{B^-}^{\GL_2} z^{\lambda^1}(\phi_i |\cdot|^{-1} \boxtimes \phi_j))^{\an}$, and $I_{\sF_i,0}:=(\Ind_{B^-}^{\GL_2} z^{\lambda^1}(\phi_i |\cdot|^{-1} \boxtimes \phi_j))^{\lalg}\cong (\Ind_{B^-}^{\GL_2} \phi_i |\cdot|^{-1}\boxtimes \phi_j)^{\infty} \otimes_E L(\lambda^1)$. We have $I_{\sF_1,0}\cong I_{\sF_2,0}=:\pi_{\alg}(D_1)$. We will use the notation $I_{\sF_i,0}$ when we want to emphasize its relation with $I_{\sF_i}$. 
	The representation $\pi_{\alg}(D_1)$ is in fact the  locally algebraic subrepresentation of $\pi(D_1)$. As $D_1$ is non-split, $\pi(D_1)\cong I_{\sF_1}\oplus_{\pi_{\alg}(D_1)} I_{\sF_2}$. Letting  $ \sC(s,\sF_i) := (\Ind_{B^-}^{\GL_2} z^{s\cdot \lambda^1}(\phi_i |\cdot|^{-1} \boxtimes \phi_j))^{\an}$ for $i,j\in \{1,2\}$, $i\neq j$ (recall $\sF_i$ denotes a refinement of $D_1$), then $\pi(D_1)$  has the form 
	$\big[\pi_{\alg}(D_1) \lin \sC(s,\sF_1) \oplus \sC(s,\sF_2)\big]$.
	
	Consider $(\Ind_{P_1^-}^{\GL_3} (\pi(D_1) \otimes_E \varepsilon^{-1}\circ \dett) \boxtimes \phi_3 z^{h_3})^{\an}$ which is isomorphic to
	\begin{equation*}
		(\Ind_{B^-}^{\GL_3} \jmath(\ul{\phi}) z^{\lambda})^{\an} \oplus_{(\Ind_{P_1^-}^{\GL_3} (\pi_{\alg}(D_1) \otimes_E \varepsilon^{-1}\circ \dett) \boxtimes \phi_3 z^{h_3})^{\an}}(\Ind_{B^-}^{\GL_3} \jmath(s_1(\ul{\phi}))z^{\lambda})^{\an}.
	\end{equation*}
	By \cite[Thm.]{OS},  we see that the constituents $\sC\in \cS$, that the induction contains, are exactly $\cS_{\sF}:=\{\sC(s_1, 1), \sC(s_1, s_1) ,\sC(s_2,1)\}$, each with multiplicity one. Moreover, it 
	contains a unique subrepresentation $\pi_1(\ul{\phi},\lambda)^-$, that is an extension of $\sC(s_1, 1) \oplus \sC(s_1, s_1) \oplus \sC(s_2,1)$ by $\pi_{\alg}(\ul{\phi},\lambda)$, with each sub-extension $[\pi_{\alg}(\ul{\phi},\lambda) \lin \sC]$ non-split.

	Similarly, let $\pi(C_1)$ be the locally analytic representation of $\GL_2(\Q_p)$ associated to $C_1$, and $\cS_{\sG}:=\{\sC(s_2,s_2), \sC(s_2, s_2s_1), \sC(s_1, s_1s_2)\}$. Then $\sC\in \cS$ appears in  $(\Ind_{P_2^-}^{\GL_3} \phi_3 z^{h_1-2}|\cdot |^{-2} \boxtimes \pi(C_1))^{\an}$ if and only if $\sC \in \cS_{\sG}$. And if so, it has multiplicity one. Moreover, $(\Ind_{P_2^-}^{\GL_3} \phi_3 z^{h_1-2}|\cdot |^{-2} \boxtimes \pi(C_1))^{\an}$ contains a unique subrepresentation $\pi_1(\ul{\phi},\lambda)^+$, that is an extension of $\oplus_{\sC\in \cS_{\sG}} \sC$ by $\pi_{\alg}(\ul{\phi}, \lambda)$, with each subextension $[\pi_{\alg}(\ul{\phi},\lambda)\lin \sC]$ non split. By the discussions, we have
	\begin{lemma}
		$\pi_1(\ul{\phi},\lambda)\cong \pi_1(\ul{\phi},\lambda)^+ \oplus_{\pi_{\alg}(\ul{\phi},\lambda)} \pi_1(\ul{\phi},\lambda)^-$.
	\end{lemma}
	\subsection{Extension groups}
	We collect some facts on the extension groups of locally analytic representations of $\GL_2(\Q_p)$ and of $\GL_3(\Q_p)$. We invite the reader to compare these with the extension groups of $(\varphi, \Gamma)$-modules considered in \S~\ref{S1}.
	
	\subsubsection{$\GL_2(\Q_p)$.}
	The following proposition  is well-known (e.g. by an easier variation of the arguments in \cite[\S~3.2.2]{BD1} or the proof of Proposition \ref{PExt2} below). 
	\begin{proposition}\label{PExt1}(1) We have a natural  exact sequence
		\begin{multline*}
			0 \lra \Ext^1_{\GL_2(\Q_p)}(\pi_{\alg}(D_1), \pi(D_1)) \lra \Ext^1_{\GL_2(\Q_p)}(\pi_{\alg}(D_1), \pi(D_1)) \\
			\lra \oplus_{i=1}^2\Ext^1_{\GL_2(\Q_p)}(\pi_{\alg}(D_1), \sC(s,\sF_i)) \lra 0.
		\end{multline*}
		Moreover,  we have $\dim_E \Ext^1_{\GL_2(\Q_p)}(\pi_{\alg}(D_1), \pi_{\alg}(D_1))=3$, $\dim_E \Ext^1_{\GL_2(\Q_p)}(\pi_{\alg}(D_1), \pi(D_1)) =5$, and $\dim_E \Ext^1_{\GL_2(\Q_p)}(\pi_{\alg}(D_1), \sC(s,\sF_i))=1$.

		(2) For $i\in \{1,2\}$, there is a natural exact sequence
		\begin{equation*}
			0 \ra \Ext^1_{\GL_2(\Q_p)}(\pi_{\alg}(D_1),\pi_{\alg}(D_1)) \ra \Ext^1_{\GL_2(\Q_p)}(\pi_{\alg}(D_1),  I_{\sF_i}) 
			\ra \Ext^1_{\GL_2(\Q_p)}(\pi_{\alg}(D_1), \sC(s, \sF_i)) \ra 0,
		\end{equation*}
	hence $\dim_E \Ext^1_{\GL_2(\Q_p)}(\pi_{\alg}(D_1),  I_{\sF_i}) =4$.
	\end{proposition}
For $i\in \{1,2\}$, consider the following composition (recalling $\delta_{\sF_i}=\phi_i z^{h_1} \boxtimes \phi_j z^{h_2}$ for $j\in \{1,2\}$, $j\neq i$)
\begin{equation}\label{EindGL2}
\zeta_i:		\Ext^1_{T(\Q_p)}(\delta_{\sF_i}, \delta_{\sF_i}) \lra \Ext^1_{\GL_2(\Q_p)}(I_{\sF_i}, I_{\sF_i}) \lra \Ext^1_{\GL_2(\Q_p)}(\pi_{\alg}(D_1), I_{\sF_i}),
\end{equation}
where the first map is obtained by applying the functor $(\Ind_{B^-}^{\GL_2}- \otimes_E (\varepsilon^{-1} \boxtimes 1))^{\an}$ and the second is the pull-back map (with respect to a fixed embedding $\pi_{\alg}(D_1)\hookrightarrow I_{\sF_i}$). Using Schraen's spectral sequence \cite{Sch11}, it is not difficult to see the composition is bijective. The composition can also be obtained by Emertons's functor $I_{B^-}^{\GL_2}$ in \cite{Em11}. Recall for $\widetilde{\delta}_{\sF_i}\in \Ext^1_{T(\Q_p)}(\delta_{\sF_i}, \delta_{\sF_i})$, $I_{B^-}^{\GL_2}(\widetilde{\delta}_{\sF_i}(\varepsilon^{-1}\boxtimes 1))$ is the closed subrepresentation of $(\Ind_{B^-}^{\GL_2} \widetilde{\delta}_{\sF_i}(\varepsilon^{-1} \boxtimes 1))^{\an}$ generated by (where $\delta_B$ denotes the modulus character of $B(\Q_p)$)
\begin{equation*}
	\widetilde{\delta}_{\sF_i}(\varepsilon^{-1} \boxtimes 1)\delta_B \hooklongrightarrow J_B\big((\Ind_{B^-}^{\GL_2} \widetilde{\delta}_{\sF_i}(\varepsilon^{-1} \boxtimes 1))^{\an}\big) \hooklongrightarrow (\Ind_{B^-}^{\GL_2} \widetilde{\delta}_{\sF_i}(\varepsilon^{-1} \boxtimes 1))^{\an}.
\end{equation*} There is a natural surjection $I_{B^-}^{\GL_2} \big(\widetilde{\delta}_{\sF_i}(\varepsilon^{-1} \boxtimes 1)\big)\twoheadrightarrow  I_{\sF_i,0} \cong \pi_{\alg}(D_1)$, and let $W$ be its kernel, which is clearly a closed subrepresentation of $I_{\sF_i}$. Then the image of $\widetilde{\delta}_{\sF_i}$ under  (\ref{EindGL2}) is just the push-forward of  $I_{B^-}^{\GL_2} \big(\widetilde{\delta}_{\sF_i}(\varepsilon^{-1} \boxtimes 1)\big)$ along $W\hookrightarrow I_{\sF_i}$. Finally, the inverse of (\ref{EindGL2}) can be described using the Jacquet-Emerton functor: for any extension $\pi\in \Ext^1_{\GL_2(\Q_p)}(\pi_{\alg}(D_1), I_{\sF_i})$,  its Jacquet-Emerton module $J_B(\pi)$ (cf. \cite{Em11}) contains a unique self-extension of $\delta_{\sF_i}(\varepsilon^{-1} \boxtimes 1) \delta_B$, which is just the inverse image of $\pi$ twisted by $(\varepsilon^{-1} \boxtimes 1) \delta_B$. 

Denote by $\Ext^1_{g'}(\delta_{\sF_i},\delta_{\sF_i})\subset \Ext^1_{T(\Q_p)}(\delta_{\sF_i}, \delta_{\sF_i})$ the subspace of extensions which are locally algebraic up to twist by a certain character of $Z(\Q_p)$ (over $E[\epsilon]/\epsilon^2$). Then $\zeta_i$ restricts to a bijection
\begin{equation*}
	\zeta_i: \Ext^1_{g'}(\delta_{\sF_i},\delta_{\sF_i}) \xlongrightarrow{\sim} \Ext^1_{\GL_2(\Q_p)}(\pi_{\alg}(D_1), \pi_{\alg}(D_1)).
\end{equation*}
Let $\xi$ denote the map 
\begin{eqnarray*}
\xi:	\Ext^1_{T(\Q_p)}(\delta_{\sF_1}, \delta_{\sF_1}) &\xlongrightarrow{\sim}& \Ext^1_{T(\Q_p)}(\delta_{\sF_2}, \delta_{\sF_2})\\
(\phi_1z^{h_1}(1+\psi_{1} \epsilon)) \boxtimes (\phi_2z^{h_2}(1+\psi_{2}\epsilon)) &\mapsto& (\phi_2z^{h_1}(1+\psi_{2} \epsilon)) \boxtimes (\phi_1z^{h_1}(1+\psi_{1}\epsilon)).
\end{eqnarray*}
By the classical intertwining, we have 
\begin{equation}\label{Eint00}
	\zeta_1 =\zeta_2 \circ \xi : \Ext^1_{g'}(\delta_{\sF_1},\delta_{\sF_1}) \xlongrightarrow{\sim} \Ext^1_{\GL_2(\Q_p)}(\pi_{\alg}(D_1), \pi_{\alg}(D_1)).
\end{equation}

Denote by $\Ext^1_{\sF_i}(\pi_{\alg}(D_1), \pi(D_1))$ \big(resp.  $\Ext^1_{g'}(\pi_{\alg}(D_1), \pi(D_1))$, resp,  $\Ext^1_{g}(\pi_{\alg}(D_1), \pi(D_1))$\big) the \ image \ via \ the \ push-forward (injective) \ map \ of  $\Ext^1_{\GL_2(\Q_p)}(\pi_{\alg}(D_1), I_{\sF_i})$ \big(resp.  of $\Ext^1_{\GL_2(\Q_p)}(\pi_{\alg}(D_1), \pi_{\alg}(D_1))$, resp.  of $\Ext^1_{\lalg}(\pi_{\alg}(D_1), \pi_{\alg}(D_1))$, ``$\lalg$" denoting the locally algebraic extensions\big). Then we have an exact sequence (compare with Proposition \ref{PGL21} (2))
\begin{equation}\label{Efern1}
	0 \ra \Ext^1_{g'}(\pi_{\alg}(D_1), \pi(D_1)) \ra \oplus_{i=1,2}\Ext^1_{\sF_i}(\pi_{\alg}(D_1), \pi(D_1))\ra \Ext^1_{\GL_2(\Q_p)}(\pi_{\alg}(D_1), \pi(D_1)) \ra 0.
\end{equation}

	\subsubsection{$\GL_3(\Q_p)$.}
	In this section, to distinguish, we use $Z$, $T$, $B$ to denote subgroups of $\GL_3$, and $Z_1$, $T_1$, $B_1$ the corresponding subgroups of $\GL_2$. 
	\begin{proposition}\label{PExt2}
		(1)  There is a natural exact sequence 
		\begin{multline} \label{EdiviGL3}
			0 \lra \Ext^1_{\GL_3(\Q_p)}(\pi_{\alg}(\ul{\phi}, \lambda), \pi_{\alg}(\ul{\phi},\lambda)) \lra \Ext^1_{\GL_3(\Q_p)}(\pi_{\alg}(\ul{\phi}, \lambda), \pi_1(\ul{\phi},\lambda))\\
			 \lra \oplus_{\sC \in \cS} \Ext^1_{\GL_3(\Q_p)}(\pi_{\alg}(\ul{\phi},\lambda), \sC) \lra 0,
		\end{multline}
		where $\dim_E \Ext^1_{\GL_3(\Q_p)}(\pi_{\alg}(\ul{\phi}, \lambda), \pi_{\alg}(\ul{\phi},\lambda))=4$, $\dim_E \Ext^1_{\GL_3(\Q_p)}(\pi_{\alg}(\ul{\phi},\lambda), \sC) =1$ for $\sC\in \cS$, and (hence) $\dim_E \Ext^1_{\GL_3(\Q_p)}(\pi_{\alg}(\ul{\phi},\lambda),\pi_1(\ul{\phi},\lambda))=10$.
		
		(2) There  is an exact sequence (where the same holds with $-$, $\cS_{\sF}$ replaced by $+$, $\cS_{\sG}$)
		\begin{multline*}
			0 \lra \Ext^1_{\GL_3(\Q_p)}(\pi_{\alg}(\ul{\phi}, \lambda), \pi_{\alg}(\ul{\phi},\lambda)) \lra \Ext^1_{\GL_3(\Q_p)}(\pi_{\alg}(\ul{\phi}, \lambda), \pi_1(\ul{\phi},\lambda)^-)\\
			 \lra \oplus_{\sC \in \cS_{\sF}} \Ext^1_{\GL_3(\Q_p)}(\pi_{\alg}(\ul{\phi},\lambda), \sC) \lra 0,
		\end{multline*}
	with $\dim_E \Ext^1_{\GL_3(\Q_p)}(\pi_{\alg}(\ul{\phi}, \lambda), \pi_1(\ul{\phi},\lambda)^-)=7$. 
		In particular, 
		$$\Ext^1_{\GL_3}(\pi_{\alg}(\ul{\phi},\lambda), \pi_{\alg}(\ul{\phi},\lambda))\xrightarrow{\sim} \Ext^1_{\GL_3}(\pi_{\alg}(\ul{\phi}, \lambda), \pi_1(\ul{\phi},\lambda)^+) \cap  \Ext^1_{\GL_3}(\pi_{\alg}(\ul{\phi}, \lambda), \pi_1(\ul{\phi},\lambda)^-).$$	
			
			(3) For $w\in S_3$, $\dim_E\Ext^1_{\GL_3(\Q_p)}(\pi_{\alg}(\ul{\phi},\lambda), \pi_1(\ul{\phi},\lambda,w))=6$, and there is an exact sequence
			\begin{multline*}
				0 \lra \Ext^1_{\GL_3(\Q_p)}(\pi_{\alg}(\ul{\phi},\lambda), \pi_{\alg}(\ul{\phi},\lambda)) \lra \Ext^1_{\GL_3(\Q_p)}(\pi_{\alg}(\ul{\phi},\lambda),\pi_1(\ul{\phi},\lambda,w)) 
				\\
				\lra \oplus_{i=1,2} \Ext^1_{\GL_3(\Q_p)}(\pi_{\alg}(\ul{\phi},\lambda),\sC(s_i,w))\lra 0.
			\end{multline*}
	\end{proposition}

	\begin{proof}
All the sequences are obtained by d\'evissage. And it suffices to show the last maps are surjective. We use $\Ext^i_Z$ to denote the subgroup of extensions with a fixed central character. 		By \cite[\S~4.3]{Sch11}, $\Ext^i_Z(\pi_{\alg}(\ul{\phi},\lambda), \pi_{\alg}(\ul{\phi},\lambda))\cong \Ext^i_{Z,\lalg}(\pi_{\alg}(\ul{\phi},\lambda), \pi_{\alg}(\ul{\phi},\lambda))$ for $i\leq 2$. Hence \begin{equation}\label{Edimalg}\dim_E \Ext^i_Z(\pi_{\alg}(\ul{\phi},\lambda), \pi_{\alg}(\ul{\phi},\lambda))=\begin{cases} 2 & i=1 \\ 0 & i=2\end{cases}.\end{equation}
By d\'evissage, we get an exact sequence
	\begin{equation*}
	0 \ra \Ext^1_{Z}(\pi_{\alg}(\ul{\phi}, \lambda), \pi_{\alg}(\ul{\phi},\lambda)) \ra \Ext^1_{Z}(\pi_{\alg}(\ul{\phi}, \lambda), \pi_1(\ul{\phi},\lambda))
	\ra \oplus_{\sC \in \cS} \Ext^1_{Z}(\pi_{\alg}(\ul{\phi},\lambda), \sC) \ra 0.
\end{equation*}
By an easier variation of the proof of  \cite[Lem.~2.28]{Ding7}, for $\sC\in \cS$, $$\dim_E \Ext^1_{Z}(\pi_{\alg}(\ul{\phi},\lambda),\sC)=\dim_E \Ext^1_{\GL_3(\Q_p)}(\pi_{\alg}(\ul{\phi},\lambda),\sC)=1.$$ Together with the commutative diagram
\begin{equation*}
	\begin{CD}
		\Ext^1_{Z}(\pi_{\alg}(\ul{\phi}, \lambda), \pi_1(\ul{\phi},\lambda))
		@>>> \oplus_{\sC \in \cS} \Ext^1_{Z}(\pi_{\alg}(\ul{\phi},\lambda), \sC) \\
		@VVV @VVV \\
		\Ext^1_{\GL_3(\Q_p)}(\pi_{\alg}(\ul{\phi}, \lambda), \pi_1(\ul{\phi},\lambda))
		@>>> \oplus_{\sC \in \cS} \Ext^1_{\GL_3(\Q_p)}(\pi_{\alg}(\ul{\phi},\lambda), \sC),
	\end{CD}
\end{equation*}
(the surjectivity in)  (\ref{EdiviGL3}) follows. By (\ref{Edimalg}) and similar arguments in \cite[Lem.~3.16]{BD1}, $$\dim_E \Ext^1_{\GL_3(\Q_p)}(\pi_{\alg}(\ul{\phi},\lambda), \pi_{\alg}(\ul{\phi},\lambda))=4.$$ (1) follows. 
 (2) and (3) follow by similar arguments. 
	\end{proof}

	\begin{remark}\label{RGL3tri}We have
		\begin{multline*}
		\zeta_w:	\Ext^1_{T(\Q_p)}\big(\jmath(w(\ul{\phi}))z^{\lambda},\jmath(w(\ul{\phi}))z^{\lambda}\big)\xlongrightarrow{\Ind}  \Ext^1_{\GL_3(\Q_p)}\big((\Ind_{B^-}^{\GL_3} \jmath(w(\ul{\phi}))z^{\lambda})^{\an},(\Ind_{B^-}^{\GL_3} \jmath(w(\ul{\phi}))z^{\lambda})^{\an}\big)\\
			\lra \Ext^1_{\GL_3(\Q_p)}\big(\pi_{\alg}(\ul{\phi},\lambda), (\Ind_{B^-}^{\GL_3} \jmath(w(\ul{\phi}))z^{\lambda})^{\an}\big) \xlongleftarrow{\sim}\Ext^1_{\GL_3(\Q_p)}(\pi_{\alg}(\ul{\phi},\lambda),\pi_1(\ul{\phi},\lambda,w))
		\end{multline*}
		where the composition of the first two maps is an isomorphism by \cite[(4.38)]{Sch11}, and the third isomorphism follows from \cite[Lem.~2.26]{Ding7}.
Similarly as in the discussion below Proposition \ref{PExt1},  	the inverse of $\zeta_w$ can be described by applying Jacquet-Emerton functor. 
	\end{remark}
	For $w\in S_3$, we put $\Ext^1_{w}(\pi_{\alg}(\ul{\phi},\lambda), \pi_1(\ul{\phi},\lambda))$ to be the image of 
\begin{equation*}
	\Ext^1_{\GL_3(\Q_p)}(\pi_{\alg}(\ul{\phi},\lambda), \pi_1(\ul{\phi},\lambda,w)) \hooklongrightarrow \Ext^1_{\GL_3(\Q_p)}(\pi_{\alg}(\ul{\phi},\lambda), \pi_1(\ul{\phi},\lambda)).
\end{equation*}
By Remark \ref{RGL3tri}, we have a natural isomorphism
\begin{equation*}
	\zeta_w: \Ext^1_{T(\Q_p)}\big(\jmath(w(\ul{\phi}))z^{\lambda},\jmath(w(\ul{\phi}))z^{\lambda}\big) \xlongrightarrow{\sim} \Ext^1_{w}(\pi_{\alg}(\ul{\phi},\lambda), \pi_1(\ul{\phi},\lambda)).
\end{equation*}
	We denote by $\Ext^1_{\sF}(\pi_{\alg}(\ul{\phi},\lambda), \pi_1(\ul{\phi},\lambda))$  (resp. $\Ext^1_{\sG}(\pi_{\alg}(\ul{\phi},\lambda), \pi_1(\ul{\phi},\lambda))$) the image of 
	\begin{equation*}
		i^- : \Ext^1_{\GL_3(\Q_p)}(\pi_{\alg}(\ul{\phi},\lambda), \pi_1(\ul{\phi},\lambda)^-)\hooklongrightarrow\Ext^1_{\GL_3(\Q_p)}(\pi_{\alg}(\ul{\phi},\lambda), \pi_1(\ul{\phi},\lambda))\end{equation*}
	\begin{equation*}
		\big(\text{resp. } 	i^+ : \Ext^1_{\GL_3(\Q_p)}(\pi_{\alg}(\ul{\phi},\lambda), \pi_1(\ul{\phi},\lambda)^+)\hooklongrightarrow\Ext^1_{\GL_3(\Q_p)}(\pi_{\alg}(\ul{\phi},\lambda), \pi_1(\ul{\phi},\lambda))\big).
	\end{equation*}
	By  Schraen's spectral sequence \cite[(4.38)]{Sch11}, there is a natural isomorphism
	\begin{multline}\label{Epar2}
		\Ext^1_{L_1(\Q_p)}\big((\pi_{\alg}(D_1) \otimes_E \varepsilon^{-1} \circ \dett) \boxtimes \phi_3 z^{h_3}, (\pi(D_1) \otimes_E \varepsilon^{-1} \circ \dett) \boxtimes \phi_3 z^{h_3}\big) \\ \xlongrightarrow{\sim} \Ext^1_{\GL_3(\Q_p)}\big(\pi_{\alg}(\ul{\phi}, \lambda), (\Ind_{P_1^-}^{\GL_3} (\pi(D_1) \otimes_E \varepsilon^{-1} \circ \dett) \boxtimes \phi_3 z^{h_3})^{\an}\big) \\
		\xlongrightarrow{\sim}  \Ext^1_{\GL_3(\Q_p)}(\pi_{\alg}(\ul{\phi},\lambda), \pi_1(\ul{\phi},\lambda)^-),
	\end{multline}
	where the first map is given by $(\Ind_{P_1^-}^{\GL_3}-)^{\an}$ composed with the pull-back map, and the second map is the restriction map, being an isomorphism by \cite[Lem.~2.26]{Ding7}. There is a natural bijection 
	\begin{multline}\label{Etwist}
		\Ext^1_{\GL_2(\Q_p)}(\pi_{\alg}(D_1), \pi(D_1)) 
		\times \Hom(\Q_p^{\times},E)\xlongrightarrow{\sim} \\	\Ext^1_{L_1(\Q_p)}\big((\pi_{\alg}(D_1) \otimes_E \varepsilon^{-1} \circ \dett) \boxtimes \phi_3 z^{h_3}, (\pi(D_1) \otimes_E \varepsilon^{-1} \circ \dett) \boxtimes \phi_3 z^{h_3}\big),
	\end{multline}
sending $(\pi, \psi)$ to $\pi\otimes_{E[\epsilon]/\epsilon^2}  \phi_3z^{h_3}(1+\psi \epsilon)$ (noting the elements in both sides can be equipped with a natural $E[\epsilon]/\epsilon^2$-structure). Taking the composition of (\ref{Epar2}) (\ref{Etwist}) and $i^-$, we obtain an isomorphism 
	\begin{equation*}
		j^-:	\Ext^1_{\GL_2(\Q_p)}(\pi_{\alg}(D_1), \pi(D_1)) 
		\times \Hom(\Q_p^{\times},E)\xlongrightarrow{\sim} \Ext^1_{\sF}(\pi_{\alg}(\ul{\phi},\lambda),\pi_1(\ul{\phi},\lambda)).
	\end{equation*}
	Similarly, we have an isomorphism
	\begin{equation*}
		j^+: \	\Ext^1_{\GL_2(\Q_p)}(\pi_{\alg}(C_1), \pi(C_1)) 
		\times \Hom(\Q_p^{\times},E)\xlongrightarrow{\sim} \Ext^1_{\sG}(\pi_{\alg}(\ul{\phi},\lambda),\pi_1(\ul{\phi},\lambda)).
	\end{equation*}
	For $i=1,2$, the refinement $\sF_i$ on $D_1$ (resp. $\sG_i$ on $C_1$) corresponds to a refinement, still denoted by $\sF_i$ (resp. $\sG_i$), on $D$. 
	\begin{proposition}\label{PExt3}
		For $i=1,2$, let $w\in S_3$ such that $w(\ul{\phi})$ is the refinement $\sF_i$ (resp. $\sG_i$), the map $j^-$ (resp. $j^+$) restricts to an isomorphism
		\begin{equation*}
			j^-: \Ext^1_{\sF_i}(\pi_{\alg}(D_1), \pi(D_1)) \times \Hom(\Q_p^{\times}, E) \xlongrightarrow{\sim} \Ext^1_{w}(\pi_{\alg}(\ul{\phi},\lambda), \pi_1(\ul{\phi},\lambda))
		\end{equation*}
		\begin{equation*}
			\text{\big(resp. }	j^+: \Ext^1_{\sG_i}(\pi_{\alg}(C_1), \pi(C_1)) \times \Hom(\Q_p^{\times},E)\xlongrightarrow{\sim} \Ext^1_{w}(\pi_{\alg}(\ul{\phi},\lambda), \pi_1(\ul{\phi},\lambda))\big).
		\end{equation*}
	\end{proposition}
	\begin{proof}
		By the explicit construction (noting elements on the both sides come from Borel induced representations), we obtain the maps in the proposition. By Remark \ref{RGL3tri} and (\ref{EindGL2}), we have the following commutative diagram:
		\begin{equation*}
			\begin{CD}
				\Ext^1_{T_1(\Q_p)}(\delta_{\sF_i}, \delta_{\sF_i}) \times \Hom(\Q_p^{\times}, E) @> (\zeta_i, \id) > \sim > \Ext^1_{\GL_2(\Q_p)}(\pi_{\alg}(D_1), \pi(D_1)) \times \Hom(\Q_p^{\times}, E) \\ 
				@V \sim VV @V j^-VV \\
				\Ext^1_{T(\Q_p)}(\jmath(w(\ul{\phi}))z^{\lambda}, \jmath(w(\ul{\phi}))z^{\lambda}) @> \zeta_w > \sim > \Ext^1_{w}(\pi_{\alg}(\ul{\phi}, \lambda), \pi_1(\ul{\phi}, \lambda)),
			\end{CD}\end{equation*}
		where the left vertical map sends $(\delta_{\sF_i}((1+\psi_1 \epsilon) \boxtimes (1+\psi_2 \epsilon)), \psi_3)$ to $\jmath(w(\ul{\phi}))z^{\lambda}((1+\psi_1 \epsilon) \boxtimes (1+\psi_2 \epsilon) \boxtimes (1+\psi_3 \epsilon))$. The map $j^+$ also sits in a similar diagram. The proposition follows.
	\end{proof}

	\subsection{$p$-adic Langlands correspondence for  $\GL_3(\Q_p)$ in the crystabelline case}We construct a $\GL_3(\Q_p)$-representation that determines and depends on $D$. 
	\subsubsection{$p$-adic local Langlands correspondence for $\GL_2(\Q_p)$ revisited}
	We collect some facts on the $p$-adic Langlands correspondence for $\GL_2(\Q_p)$.

	Let $D_1$ be as in \S~\ref{S1}, and recall $\delta_{\sF_1}=\phi_1 z^{h_1-1} \boxtimes \phi_2 z^{h_2}$, $\delta_{\sF_2}=\phi_2 z^{h_1-1} \boxtimes \phi_1 z^{h_2}$. We fix $\pi_{\alg}(D_1)\hookrightarrow \pi(D_1)$ and $\pi_{\alg}(D_1)\hookrightarrow I_{\sF_i}$ (hence $I_{\sF_i}\hookrightarrow \pi(D_1)$). We have isomorphisms
	\begin{equation*}
		\rec_i: \Ext^1_{\sF_i}(D_1, D_1) \xlongrightarrow{\sim} \Ext^1_{T(\Q_p)}(\delta_{\sF_i}, \delta_{\sF_i}) \xlongrightarrow[\sim]{\zeta_i}  \Ext^1_{ \sF_i}(\pi_{\alg}(D_1), \pi(D_1))
	\end{equation*}
	which restricts to isomorphisms 
	\begin{equation*}
		\Ext^1_{g'}(D_1,D_1) \xlongrightarrow{\sim} \Ext^1_{g'} (\delta_{\sF_i}, \delta_{\sF_i}) \xlongrightarrow[\sim]{\zeta_i}  \Ext^1_{g,Z}(\pi_{\alg}(\ul{\phi}, \lambda), \pi(D_1)).
	\end{equation*}
The following lemma is a direct consequence of (\ref{Eint00}) and Remark \ref{Rint01}.
	\begin{lemma}\label{LdRnorm1}
		For $i=1,2$, 	$\rec_i|_{\Ext^1_{g'}(D_1,D_1)}$ are equal .
	\end{lemma}
By Lemma \ref{LdRnorm1} and Proposition \ref{PGL21} (2), $(\rec_1, \rec_2)$ ``glue" to a bijection
	\begin{equation}\label{Erec}
		\rec: \Ext^1_{(\varphi, \Gamma)}(D_1, D_1) \xlongrightarrow{\sim} \Ext^1_{\GL_2(\Q_p)}(\pi_{\alg}(D_1), \pi(D_1)).
	\end{equation}

	\subsubsection{Crystabelline correspondence for $\GL_3(\Q_p)$}
	Let $D$ be as in \S~\ref{S1}, and $\iota_D\in \Hom_{(\varphi, \Gamma)}(D_1, C_1)$ be associated to $D$, and $\sI_D$ be the associated set of higher intertwining pairs. Recall $\iota_D$ is determined by $\sI_D$. We associate to $(\ul{\phi}, \lambda, \iota_D)$ a locally analytic representation of $\GL_3(\Q_p)$. 
	
	Consider the compositions
	\begin{equation*}
		i^-: \Ext^1_{(\varphi, \Gamma)}(D_1, D_1) \xlongrightarrow[\sim]{\rec}  \Ext^1_{\GL_2(\Q_p)}(\pi_{\alg}(D_1), \pi(D_1)) \xlongrightarrow{j^-} \Ext^1_{\GL_3(\Q_p)}(\pi_{\alg}(\ul{\phi},\lambda), \pi_1(\ul{\phi},\lambda)),
	\end{equation*}
		\begin{equation*}
		i^+: \Ext^1_{(\varphi, \Gamma)}(C_1, C_1) \xlongrightarrow[\sim]{\rec}  \Ext^1_{\GL_2(\Q_p)}(\pi_{\alg}(C_1), \pi(C_1)) \xlongrightarrow{j^+} \Ext^1_{\GL_3(\Q_p)}(\pi_{\alg}(\ul{\phi},\lambda), \pi_1(\ul{\phi},\lambda)).
	\end{equation*}
	Recall we have $\Ext^1_g(D_1,D_1)\xrightarrow[\sim]{\kappa} \Hom_{\sm}(\Q_p,E)^{\oplus 2}$ and $\Ext^1_g(C_1,C_1)\xrightarrow[\sim]{\kappa} \Hom_{\sm}(\Q_p,E)^{\oplus 2}$. By the explicit description of the maps (see Remark \ref{Rint01}, Remark \ref{RGL3tri}), we have:
	\begin{lemma} \label{Lamal1}
		We have $i^- \circ \kappa^{-1}=i^+\circ \kappa^{-1}$. In particular, $i^-(\Ext^1_g(D_1,D_1))=i^+(\Ext^1_g(C_1,C_1))=:\Ext^1_{g,\phi_3}(\pi_{\alg}(\ul{\phi},\lambda), \pi_1(\ul{\phi},\lambda))$.
	\end{lemma}

	Let $\iota: D_1\hookrightarrow C_1$.
	Consider the restriction of $i^-$ and $i^+$ to $\Ext^1_{\iota}(D_1, D_1)$ and $\Ext^1_{\iota}(C_1, C_1)$.
	\begin{lemma} \label{Lamal2}
		We have $i^-(\Ext^1_{\iota}(D_1,D_1))\cap i^+(\Ext^1_{\iota}(C_1,C_1))=\Ext^1_{g,\phi_3}(\pi_{\alg}(\ul{\phi},\lambda), \pi_1(\ul{\phi},\lambda))$.
	\end{lemma}
	\begin{proof}
	By Proposition \ref{PExt2} (2), $i^-(\Ext^1_{\iota}(D_1,D_1))\cap i^+(\Ext^1_{\iota}(C_1,C_1)) \subset \Ext^1_{\GL_3(\Q_p)}(\pi_{\alg}(\ul{\phi},\lambda),\pi_{\alg}(\ul{\phi},\lambda))$ hence is contained in $\Ext^1_w(\pi_{\alg}(\ul{\phi},\lambda), \pi_1(\ul{\phi},\lambda))$ for all $w\in S_3$. The lemma then follows from Proposition \ref{PExt3},  Proposition \ref{Ppairing} (2) and Lemma \ref{Lamal2}.
	\end{proof}
	 Let $\sI_{\iota}^0\subset \sI_{\iota}$ be the subset consisting of non-de Rham pairs. Let $(\widetilde{D}_1, \widetilde{C}_1) \in \sI_{\iota}^0$, and put $\pi(\ul{\phi}, \lambda, \iota)\in \Ext^1_{\GL_3(\Q_p)}(\pi_{\alg}(\ul{\phi},\lambda), \pi_1(\ul{\phi},\lambda))$ to be the isomorphism class of $i^-([\widetilde{D}_1])-i^+([\widetilde{C}_1])$. 
	\begin{lemma}\label{Lindep2}
		The representation $\pi(\ul{\phi}, \lambda, \iota)$ does not depend on the choice of $(\widetilde{D}_1, \widetilde{C}_1)\in \sI_{\iota}^0$.
	\end{lemma}
\begin{proof}
	Let $(\widetilde{D}_1', \widetilde{C}_1')\in \sI_{\iota}^0$. By Proposition \ref{Ppairing} (1) (2), there exist a de Rham pair $(x_0, y_0)\in \sI_{\iota}$ and $a\in E^{\times}$ such that $[\widetilde{D}_1']=[x_0]+a[\widetilde{D}_1]$ and $[\widetilde{C}_1']=[y_0]+a[\widetilde{C}_1]$.  By Proposition \ref{Pamal1} and Lemma \ref{Lamal1}, $i^-([x_0])=i^+[y_0]$, hence $i^-([\widetilde{D}_1'])-i^+([\widetilde{C}_1'])=a(i^-([\widetilde{D}_1])-i^+([\widetilde{C}_1]))$. The lemma follows. 
\end{proof}
	By definition, $\pi(\ul{\phi},\lambda, \iota)$ has the following structure
	\begin{equation*}	\bigg[\begindc{\commdiag}[100] 
		\obj(0,0)[a]{$\pi_{\alg}(\ul{\phi},\lambda)$}
		\obj(8,3)[b]{$\sC_{\sG}(\ul{\phi},\lambda)$}
		\obj(8,-3)[c]{$\sC_{\sF}(\ul{\phi},\lambda)$}
		\obj(16,0)[d]{$\pi_{\alg}(\ul{\phi},\lambda)$}
		\mor{a}{b}{}[+1,3]
		\mor{a}{c}{}[+1,3]
		\mor{d}{b}{}[+0,3]
		\mor{d}{c}{}[+1,3]
		\enddc\bigg] \ =\ 
		\bigg[	\begindc{\commdiag}[100] 
		\obj(0,0)[a]{$\pi_{\alg}(\ul{\phi},\lambda)$}
		\obj(8,-5)[b]{$\sC(s_1,1)$}
		\obj(8,-3)[c]{$\sC(s_1,s_1)$}
		\obj(8,-1)[d]{$\sC(s_1,s_1s_2)$}
		\obj(8,1)[e]{$\sC(s_2, s_2s_1)$}
		\obj(8,3)[f]{$\sC(s_2,s_2)$}
		\obj(8,5)[g]{$\sC(s_2,1)$}
		\obj(16,0)[h]{$\pi_{\alg}(\ul{\phi},\lambda)$}
		\mor{a}{b}{}[+1,3]
		\mor{a}{c}{}[+1,3]
		\mor{a}{d}{}[+1,3]
		\mor{a}{e}{}[+1,3]
		\mor{a}{f}{}[+1,3]
		\mor{a}{g}{}[+1,3]
		\mor{h}{b}{}[+1,3]
		\mor{h}{c}{}[+1,3]
		\mor{h}{d}{}[+1,3]
		\mor{h}{e}{}[+1,3]
		\mor{h}{f}{}[+1,3]
		\mor{h}{g}{}[+1,3]
		\enddc\bigg]
	\end{equation*}
	The representation $\pi(\ul{\phi}, \lambda, \iota)$ has the following ``coordinate-free"  description.
	\begin{proposition}\label{Pinter}
		We have $\pi(\ul{\phi},\lambda,\iota)=\cap_{(\widetilde{D}_1,\widetilde{C}_1) \in \sI_{\iota}^0} \Big(i^-(\widetilde{D}_1) \oplus_{\pi_1(\ul{\phi},\lambda)} i^+(\widetilde{C}_1)\Big)$, where the intersection is taken in the universal extension of $
		\Ext^1_{\GL_3(\Q_p)}(\pi_{\alg}(\ul{\phi},\lambda), \pi_1(\ul{\phi},\lambda)) \otimes_E \pi_{\alg}(\ul{\phi},\lambda)$ by $\pi_1(\ul{\phi},\lambda)$, and we also use $i^{\pm}(-)$ to denote the representation corresponding to the elements in the extension group.
	\end{proposition}
	\begin{proof}
	The  ``$\subset$" part follows from Lemma \ref{Lindep2}. For ``$\supset$",   let  $(\widetilde{D}_1, \widetilde{C}_1)\in \sI_{\iota}^0$, $(x_0, y_0)\in \sF_{\iota}$ be de Rham (and non-zero) and $(\widetilde{D}_1', \widetilde{C}_1') \in \sF_{\iota}^0$ such that $[\widetilde{D}_1']=[\widetilde{D}_1]+[x_0]$, $[\widetilde{C}_1']=[\widetilde{C}_1]+[y_0$. If $$\big(i^-(\widetilde{D}_1) \oplus_{\pi_1(\ul{\phi},\lambda)} i^+(\widetilde{C}_1)\big) \cap \big(i^-(\widetilde{D}_1') \oplus_{\pi_1(\ul{\phi},\lambda)} i^+(\widetilde{C}_1')\big)\supsetneq \pi(\ul{\phi},\lambda, \iota),$$  then $\big(i^-(\widetilde{D}_1) \oplus_{\pi_1(\ul{\phi},\lambda)} i^+(\widetilde{C}_1)\big) \cong \big(i^-(\widetilde{D}_1') \oplus_{\pi_1(\ul{\phi},\lambda)} i^+(\widetilde{C}_1')\big)$. However, by Proposition \ref{PExt2} (2), Lemma \ref{Lamal1} and \ref{Lamal2}, $i^-([\widetilde{D}_1])$, $i^+([\widetilde{C}_1])$ and $i^-([x_0])=i^+([y_0])$ are linearly independent, a contradiction. 
	\end{proof}
	\begin{theorem}
		For two injections $\iota, \iota': D_1 \hookrightarrow C_1$, $\pi(\ul{\phi}, \lambda, \iota)\cong \pi(\ul{\phi}, \lambda, \iota')$ if and only if $E[\iota]=E[\iota']\subset  \Hom_{(\varphi, \Gamma)}(D_1,C_1)$. In particular, for $D$ given as in \S~\ref{S1}, $\pi(\ul{\phi},\lambda, \iota_D)$ depends on and uniquely determines $D$.
	\end{theorem}
	\begin{proof}
		Suppose $\pi(\ul{\phi}, \lambda, \iota)\cong \pi(\ul{\phi}, \lambda, \iota')$, and $E[\iota] \neq E[\iota']$. Let $(x, y)\in \sI_{\iota}^0$, and $(x',y')\in \sI_{\iota'}^0$ such that $[\pi(\ul{\phi}, \lambda, \iota)]=i^-(x)-i^+(y)$ and $[\pi(\ul{\phi},\lambda,\iota')]=i^-(x')-i^+(y')$. Hence $i^-(x-x')=i^+(y'-y)$. Similarly as in the proof of Lemma \ref{Lamal2}, by Proposition \ref{PExt2} (2), this implies $i^-(x-x')\in \Ext^1_{\GL_3(\Q_p)}(\pi_{\alg}(\ul{\phi},\lambda), \pi_{\alg}(\ul{\phi},\lambda))$. By  Proposition \ref{PExt3} (and the proof, see also Remark \ref{RGL3tri}), this implies  $x-x'\in \Ext^1_g(D_1,D_1)$ hence $\Ext^1_{\iota}(D_1,D_1)=\Ext^1_{\iota'}(D_1,D_1)$, contradicting Proposition \ref{Ppairing} (3).
	\end{proof}
	\begin{remark}
		(1) 	If $D$ is critical for some refinements or if $D$ has irregular Sen weights, one can associate to $D$ a semi-simple locally analytic representation of $\GL_3(\Q_p)$ that determines $D$ as in \cite{Br13I} \cite{Wu21}. Together with our construction of $\pi(\ul{\phi}, \lambda, \iota)$, this establishes a one-to-one correspondence between generic crystabelline $(\varphi, \Gamma)$-modules of rank $3$ and their corresponding locally analytic representations of $\GL_3(\Q_p)$.

		(2) Note that $\pi(\ul{\phi},\lambda, \iota_D)$ is  only a small piece of the hypothetical full locally analytic $\GL_3(\Q_p)$-representation $\pi(D)$ associated to $D$. We quickly discuss the relation between $\pi(\ul{\phi},\lambda, \iota_D)$ with the wall-crossing of $\pi(D)$ considered in \cite{Ding15}. Let $\theta_1:=(h_1-2, h_1-2,h_3)$ which is a partially singular weight. Consider the wall-crossing $R_{s_1} \pi(D):=T_{\theta_1}^{\lambda} T_{\lambda}^{\theta_1} \pi(D)$. We have a natural map $\pi(D) \xrightarrow{j} R_{s_1} \pi(D)$. Let $\pi_{s_2}^-(D):=\Ker (j)$ and $\pi_{s_2}^+(D):=\Ima(j)$. 
		Since $D$ is non-critical for all refinements, one may expect that $\pi_{s_2}^{\pm}(D)$ only depends on $\ul{\phi}$ and $\lambda$. Indeed, one may expect $\pi_{s_2}^-(D)$ to be the extension associated to $\Fil^{\max} D_{\rm{pcr}}(D)$ in \cite[(EXT)]{Br16}. One may furthermore expect that the Hodge parameter of $D$ is encoded in $\Ext^1_{\GL_3(\Q_p)}(\pi_{s_2}^+(D), \pi_{s_2}^-(D))$. Let \begin{eqnarray*}
			\pi_{s_2}^-(\ul{\phi},\lambda)&:=&\big[\pi_{\alg}(\ul{\phi},\lambda) \lin \big(\sC(s_2,1) \oplus \sC(s_2, s_2) \oplus \sC(s_2, s_2s_1)\big)\big]\\
			\pi_{s_2}^+(\ul{\phi},\lambda)&:=&\big[\big(\sC(s_1,1) \oplus \sC(s_1, s_1) \oplus \sC(s_1, s_1s_2)\big) \lin \pi_{\alg}(\ul{\phi},\lambda)\big].
		\end{eqnarray*}
		The (conjectural) injection $\pi(\ul{\phi},\lambda, \iota_D)\hookrightarrow \pi(D)$ fits into the following  commutative diagram
		\begin{equation*}
			\begin{CD}
				0 @>>>  	\pi_{s_2}^-(\ul{\phi},\lambda) @>>> \pi(\ul{\phi},\lambda, \iota_D) @>>> 	\pi_{s_2}^+(\ul{\phi},\lambda) @>>> 0\\
				@. @VVV @VVV @VVV @. \\
				0 @>>> \pi_{s_2}^-(D) @>>> \pi(D) @>>> \pi_{s_2}^+ (D) @>>> 0.
			\end{CD}
		\end{equation*}
		The Hodge parameter of $D$,  presumed to be encoded in the mysterious  $\Ext^1_{\GL_3(\Q_p)}(\pi_{s_2}^+(D), \pi_{s_2}^-(D))$, turns out to live in the much more transparent $\Ext^1_{\GL_3(\Q_p)}(\pi_{s_2}^+(\ul{\phi},\lambda), \pi_{s_2}^-(\ul{\phi},\lambda))$.
	\end{remark}
	\section{Local-global compatibility}
	\subsection{Universal extensions}
	Keep the notation. We collect some facts on  the action of the universal deformation ring of $D_1$ on universal extensions of $\pi_{\alg}(D_1)$ by $\pi(D_1)$. Keep  fixing injections $\pi_{\alg}(D_1) \hookrightarrow I_{\sF_i} \hookrightarrow \pi(D_1)$.

	Let $R_{D_1}$ (resp. $R_{D_1,\sF_i}$) be the universal deformation ring  of deformations of $D_1$ (resp. trianguline deformations of $D_1$ with respect to $\sF_i$) over artinian local $E$-algebras. It is known that $R_{D_1}$ and $R_{D_1, \sF_i}$ are both completed formally smooth local $E$-algebras, and there are natural quotient maps $R_{D_1} \twoheadrightarrow R_{D_1, \sF_i}$. We let $\fm$ (resp. $\fm_i$) be the maximal ideal of $R_{D_1}$ (resp. of $R_{D_1, \sF_i}$). 
	
	Let $\widetilde{\pi}(D_1)^{\univ}$ be the universal extension of $\Ext^1_{\GL_2(\Q_p)}(\pi_{\alg}(D_1), \pi(D_1)) \otimes_E \pi_{\alg}(D_1)$ by $\pi(D_1)$. Then $\widetilde{\pi}(D_1)^{\univ}$ can be equipped with a natural $E$-linear action of $R_{D_1}/\fm^2$ as follows. For $x\in \fm/\fm^2$, we view $x$ as a map $\Ext^1_{\GL_2(\Q_p)}(\pi_{\alg}(D_1), \pi(D_1)) \ra E$ using 
	\begin{equation*}
		(\fm/\fm^2)^{\vee} \cong \Ext^1_{(\varphi, \Gamma)}(D_1,D_1) \xlongrightarrow{\rec} \Ext^1_{\GL_2(\Q_p)}(\pi_{\alg}(D_1), \pi(D_1)).
	\end{equation*}
	Then $x$ acts on $\widetilde{\pi}(D_1)^{\univ}$ via 
	\begin{equation}\label{Exaction}
		\widetilde{\pi}(D_1)^{\univ} \twoheadlongrightarrow \Ext^1_{\GL_2(\Q_p)}(\pi_{\alg}(D_1), \pi(D_1)) \otimes_E \pi_{\alg}(D_1) \xlongrightarrow{x} \pi_{\alg}(D_1) \hooklongrightarrow \pi(D_1) \hooklongrightarrow \widetilde{\pi}(D_1)^{\univ}.
	\end{equation}
	The  lemma below follows by definition, where ``$[\cJ]$" denotes the subspace annihilated by the ideal $\cJ$. 
	\begin{lemma}\label{Ltan1}
		(1) $\widetilde{\pi}(D_1)^{\univ}[\fm]=\pi(D_1)$.
		
		(2) Let $v\in \Spec E[\epsilon]/\epsilon^2 \ra \Spf R_{D_1}$ be a non-zero element in the tangent space of $\Spf R_{D_1}$ at $\fm$, $I_v\supset \fm^2$ be the associated ideal. Let $\widetilde{D}_{1,v}$ be the associated deformation of $D_1$ over $\cR_{E[\epsilon]/\epsilon^2}$. Then $\widetilde{\pi}(D_1)^{\univ}[I_v]\cong \rec(\widetilde{D}_{1,v})$. 
	\end{lemma}
	For our application, we  give an ``alternative" construction of the $R_{D_1}\times \GL_2(\Q_p)$-module $\widetilde{\pi}(D_1)^{\univ}$ using trianguline deformations.   Let $\fm_{\delta_{\sF_i}}$ be the maximal ideal of $\widehat{T}$ at $\delta_{\sF_i}$, and   $\widehat{\co}_{\widehat{T}, \fm_{\delta_{\sF_i}}}$ be the completion  of $\widehat{T}$ at $\fm_{\delta_{\sF_i}}$. We have a natural  isomorphism of completed $E$-algebras: $\widehat{\co}_{\widehat{T}, \fm_{\delta_{\sF_i}}} \xrightarrow{\sim} R_{D_1,\sF_i}$ sending a trianguline deformation to its corresponding trianguline parameter. Let $\widetilde{\delta}_{\sF_i}^{\univ}$ be the  the universal deformation of $\delta_{\sF_i}$ over $ \widehat{\co}_{\widehat{T}, \fm_{\delta_{\sF_i}}}/\fm_{\delta_{\sF_i}}^2\cong R_{D_1,\sF_i}/\fm_i^2$, which is a free $R_{D_1,\sF_i}/\fm^2_i$-module  of rank $1$. Consider $\widetilde{\delta}_{\sF_i}^{\univ,\vee}:=\Hom_E(\widetilde{\delta}_{\sF_i}^{\univ}, E)$, equipped with a natural action of $T(\Q_p) \times R_{D_1,\sF_i}/\fm_i^2$  given by $(af)(x)=f(ax)$ for $a\in T(\Q_p) \times R_{D_1,\sF_i}/\fm^2_i$. It is straightforward to see  $\widetilde{\delta}_{\sF_i}^{\univ, \vee}$ is  isomorphic, as $T(\Q_p)$-representation, to the universal  extension of $\Ext^1_{T(\Q_p)}(\delta_{\sF_i},\delta_{\sF_i}) \otimes_E \delta_{\sF_i}$ by $\delta_{\sF_i}$. And the $R_{D_1,\sF_i}/\fm_i^2$-action on $\widetilde{\delta}_{\sF_i}^{\univ, \vee}$  admits a similar description as in (\ref{Exaction}) (using Proposition \ref{PGL21} (3)). Consider the $\GL_2(\Q_p)$-representation $I_{B^-}^{\GL_2} \big(\widetilde{\delta}_{\sF_i}^{\univ, \vee} (\varepsilon^{-1} \boxtimes 1)\big)$, which is equipped with an induced $R_{D_1,\sF_i}/\fm_i^2$-action. 
	\begin{lemma}\label{LRmod1}
		$I_{B^-}^{\GL_2}\big(\widetilde{\delta}_{\sF_i}^{\univ, \vee} (\varepsilon^{-1} \boxtimes 1)\big)$ is the universal extension of $\Ext^1_{\GL_2(\Q_p)}(I_{\sF_i,0}, I_{\sF_i}) \otimes_E I_{\sF_i,0}$ by $I_{\sF_i}$. Moreover, the action of  $x\in \fm_i/\fm_i^2$-action is given by 
		\begin{equation*}
			I_{B^-}^{\GL_2}\big(\widetilde{\delta}_{\sF_i}^{\univ, \vee} (\varepsilon^{-1} \boxtimes 1)\big)\twoheadlongrightarrow \Ext^1_{\GL_2(\Q_p)}(I_{\sF_i,0}, I_{\sF_i}) \otimes_E I_{\sF_i,0} \xlongrightarrow{x\circ \rec_i}  I_{\sF_i,0} \hooklongrightarrow  I_{B^-}^{\GL_2}  \big(\widetilde{\delta}_{\sF_i}^{\univ, \vee} (\varepsilon^{-1} \boxtimes 1)\big).
		\end{equation*} 
	\end{lemma}
	\begin{proof}By the exact sequence 
		\begin{equation*}
			0 \lra \delta_{\sF_i}  \lra \widetilde{\delta}_{\sF_i}^{\univ, \vee} \lra \Ext^1_{T(\Q_p)}(\delta_{\sF_i}, \delta_{\sF_i}) \otimes_E \delta_{\sF_i} \lra 0, 
		\end{equation*}
		and \cite[Lem.~4.12]{Ding7}, it is not difficult to see $I_{B^-}^{\GL_2}\big(\widetilde{\delta}_{\sF_i}^{\univ, \vee} (\varepsilon^{-1} \boxtimes 1)\big)$ sits in  an exact sequence
		\begin{equation*}
			0 \lra I_{\sF_i} \lra I_{B^-}^{\GL_2}\big(\widetilde{\delta}_{\sF_i}^{\univ, \vee} (\varepsilon^{-1} \boxtimes 1)\big)  \xlongrightarrow{f} \Ext^1_{T(\Q_p)}(\delta_{\sF_i}, \delta_{\sF_i})  \otimes_E I_{\sF_i,0} \lra 0.
		\end{equation*}
		Indeed, the kernel of  $f$ is obviously a subrepresentation of $I_{\sF_i}$, which can not be $I_{\sF_i,0}$ as for $\iota_0$, $\iota_i$ in Proposition \ref{PGL21}, we have $\Ima(\iota_i)\supsetneq \Ima(\iota_0)$.
		Using the natural isomorphism $\Ext^1_{T(\Q_p)}(\delta_{\sF_i}, \delta_{\sF_i}) \cong \Ext^1_{\GL_2(\Q_p)}(I_{\sF_i,0}, I_{\sF_i})$ (cf. (\ref{EindGL2})) and the fact that $ \widetilde{\delta}_{\sF_i}^{\univ, \vee}$ is universal, the first part of the lemma follows. The second is straightforward to check.
	\end{proof}
	In particular, $I_{B^-}^{\GL_2}\big(\widetilde{\delta}_{\sF_i}^{\univ, \vee} (\varepsilon^{-1} \boxtimes 1)\big)[\fm_i]=I_{\sF_i}$. We equip $\pi(D_1)$ with an action of $R_{D_1, \sF_i}$ via $R_{D_1, \sF_i}\twoheadrightarrow R_{D_1, \sF_i}/\fm_i$. The amalgamated sum for both the action of $\GL_2(\Q_p)$ and of $R_{D_1, \sF_i}$ (recalling we have fixed $I_{\sF_i}\hookrightarrow \pi(D_1)$)
	\begin{equation}\label{Ealm1}
		\widetilde{\pi}(D_1)^{\univ}_{\sF_i}:=	I_{B^-}^{\GL_2}\big(\widetilde{\delta}_{\sF_i}^{\univ, \vee} (\varepsilon^{-1} \boxtimes 1)\big) \oplus_{I_{\sF_i}} \pi(D_1)
	\end{equation}
	is  no other than the universal extension of $\Ext^1_{\sF_i}(\pi_{\alg}(D_1), \pi(D_1)) \otimes_E\pi_{\alg}(D_1)$ by $\pi(D_1)$ (using the fixed isomorphism $\pi_{\alg}(D_1)\cong I_{\sF_i,0}$). And the $R_{D_1,\sF_i}$-action on it (via (\ref{Ealm1})) coincides with  the one given in a similar way as in (\ref{Exaction}) using $\rec_i$. 
	
	Let $\widehat{\co}^{g'}_{\widehat{T}, \fm_{\delta_{\sF_i}}}$ be the completion at $\fm_{\delta_{\sF_i}}$ of the fibre of the composition $$\widehat{T} \lra \bA^2 \xlongrightarrow{(x-y,\id)} \bA^1 $$ at $h_1-h_2$, where the first map is the weight map. Denote by $R_{D_1,g'}$ the universal deformation ring of those deformations which are de Rham up to twist by characters (over artinian $E$-algebras) with $\fm_0$ its maximal ideal. We have a natural quotient map $R_{D_1, \sF_i} \twoheadrightarrow R_{D_1,g'}$, which induces
	an isomorphism $\widehat{\co}^{g'}_{\widehat{T}, \fm_{\delta_{\sF_i}}} \xrightarrow{\sim} R_{D_1,g'}$. Denote by $\widetilde{\delta}_{\sF_i,g'}^{\univ}$ the  universal deformation of $\delta_{\sF_i}$ over $\widehat{\co}^{g'}_{\widehat{T}, \fm_{\delta_{\sF_i}}}/\fm_{\delta_{\sF_i}}^2\cong R_{D_1,g'}/\fm_0^2$. Similarly as above, its dual $\widetilde{\delta}_{\sF_i,g'}^{\univ, \vee}$  is the universal extension of $\Ext^1_{g'}(\delta_{\sF_i}, \delta_{\sF_i}) \otimes_E \delta_{\sF_i}$ by $\delta_{\sF_i}$, with the $\fm_0/\fm_0^2$-action given in a similar way as in (\ref{Exaction}). 
	\begin{lemma}\label{Ldeiso}
		(1) 	$I_{B^-}^{\GL_2} \big(\widetilde{\delta}_{\sF_i,g'}^{\univ, \vee}(\varepsilon^{-1} \boxtimes 1)\big)$ is the universal extension of $\Ext^1_{\GL_2(\Q_p)}(I_{\sF_i,0}, I_{\sF_i,0}) \otimes_E I_{\sF_i,0}$ by $I_{\sF_i,0}$. Moreover the action of $x\in \fm_0/\fm_0^2$ is given by  	\begin{equation*}
			I_{B^-}^{\GL_2}\big(\widetilde{\delta}_{\sF_i, g'}^{\univ, \vee} (\varepsilon^{-1} \boxtimes 1)\big)\twoheadlongrightarrow \Ext^1_{\GL_2(\Q_p)}(I_{\sF_i,0}, I_{\sF_i,0}) \otimes_E I_{\sF_i,0} \xlongrightarrow{x\circ \rec_i}  I_{\sF_i,0} \hooklongrightarrow I_{B^-}^{\GL_2} \big(\widetilde{\delta}_{\sF_i,g'}^{\univ, \vee}(\varepsilon^{-1} \boxtimes 1)\big).
		\end{equation*} 
		
		(2) The natural injection $I_{B^-}^{\GL_2} \big(\widetilde{\delta}_{\sF_i,g'}^{\univ, \vee}(\varepsilon^{-1} \boxtimes 1)\big) \hookrightarrow I_{B^-}^{\GL_2} \big(\widetilde{\delta}_{\sF_i}^{\univ, \vee}(\varepsilon^{-1} \boxtimes 1)\big)$ is $R_{D_1, \sF_i}\times \GL_2(\Q_p)$-equivariant.
		
		(3) We have a $\GL_2(\Q_p)\times R_{D_1,g'}/\fm_0^2$-equivariant isomorphism 
		$$I_{B^-}^{\GL_2} \big(\widetilde{\delta}_{\sF_1,g'}^{\univ, \vee}(\varepsilon^{-1} \boxtimes 1)\big) \cong I_{B^-}^{\GL_2} \big(\widetilde{\delta}_{\sF_2,g'}^{\univ, \vee}(\varepsilon^{-1} \boxtimes 1)\big).$$ 
	\end{lemma}
	\begin{proof}
		(1) follows by similar arguments as in the proof of Lemma \ref{LRmod1}. (2) is clear. (3) follows from (1), (\ref{Eint00}) and Lemma \ref{LdRnorm1}.
	\end{proof}
	We have a $\GL_2(\Q_p) \times R_{D_1}$-equivariant injection (where  the  amalgamated sum is for the $\GL_2(\Q_p)\times R_{D_1,g'}/\fm_0^2$-action, noting the fixed injection $I_{\sF_i,0}\hookrightarrow \pi(D_1)$ is obviously $R_{D_1}$-equivariant)
	\begin{equation*}
		\widetilde{\pi}(D_1)_{\sF_i,g'}^{\univ}:=I_{B^-}^{\GL_2} \big(\widetilde{\delta}_{\sF_1,g'}^{\univ, \vee}(\varepsilon^{-1} \boxtimes 1)\big) \oplus_{I_{\sF_i,0}} \pi(D_1)\hooklongrightarrow \widetilde{\pi}(D_1)_{\sF_i}^{\univ}.
	\end{equation*}
	By Lemma \ref{Ldeiso}, we have a $\GL_2(\Q_p)\times R_{D_1}$-equivariant isomorphism $\widetilde{\pi}(D_1)_{\sF_1,g'}^{\univ} \cong \widetilde{\pi}(D_1)_{\sF_2,g'}^{\univ}=: \widetilde{\pi}(D_1)_{g'}^{\univ}$. Comparing the $R_{D_1}$-action and using the construction of $\rec$ (\ref{Erec}), we obtain
	\begin{proposition}\label{Puniv2}
		We have a $\GL_2(\Q_p)\times R_{D_1}/\fm^2$-equivariant isomorphism
		\begin{equation*}
			\widetilde{\pi}(D_1)^{\univ} \cong \widetilde{\pi}(D_1)_{\sF_1}^{\univ} \oplus_{\widetilde{\pi}(D_1)_{g'}^{\univ}} \widetilde{\pi}(D_1)_{\sF_2}^{\univ}.
		\end{equation*}
	\end{proposition}
	
	\subsection{Local-global compatibility on the parabolic Jacquet-Emerton modules}
	We prove a local-global compatibility result on the parabolic Jacquet-Emerton module of the patched locally analytic representation. In this section, let $T_1$ be the torus subgroup of $\GL_2$, $B_1\supset T_1$ be the subgroup of upper triangular matrices of $\GL_2$.
	
	Let 
	$\Pi_{\infty}$ be the patched Banach representation in \cite{CEGGPS1} (for $\GL_n(F)=\GL_3(\Q_p)$), which is equipped with an action of the patched Galois deformation ring $R_{\infty}\cong R_{\overline{\rho}}^{\square} \widehat{\otimes}_{\co_E} R_{\infty}^{\wp}$ (where $\wp$ is ``$\widetilde{\fp}$" and $\overline{\rho}$ is the local Galois representation $\overline{r}$ of \textit{loc. cit.}). We refer to \textit{loc. cit.} for details. Let 
	\begin{equation}\label{Eeigen}
		\cE \hooklongrightarrow (\Spf R_{\overline{\rho}}^{\square})^{\rig} \times \widehat{T} \times (\Spf R_{\infty}^{\wp})^{\rig}
	\end{equation}be the associated patched eigenvariety (see \cite[\S~4.1.2]{Ding7}, that is an easy variation of those introduced in \cite{BHS1}), $\cM$ be the natural coherent sheaf on $\cE$ such that there is an $T(\Q_p)\times R_{\infty}$-equivariant isomorphism (see \cite[\S~3.1]{BHS1} for ``$R_{\infty}-\an$")
	\begin{equation*}
		\Gamma(\cE, \cM)^{\vee} \cong J_B(\Pi_{\infty}^{R_{\infty}-\an}).
	\end{equation*}
	Let  $X_{\tri}^{\square}(\overline{\rho})$ be the trianguline variety \cite[\S~2.2]{BHS1}. Recall (\ref{Eeigen}) factors through an embedding $\cE \hookrightarrow X_{\tri}^{\square}(\overline{\rho}) \times (\Spf R_{\infty}^{\wp})^{\rig}$ (cf. \cite[Thm.~1.1]{BHS1}, see also \cite{KPX}\cite{Liu}).  Finally, let $n_1:=\dim (\Spf R_{\infty}^{\wp})^{\rig}$, then  $\dim \cE=(3^2+\frac{3(3+1)}{2})+n_1$.
	
	Let $\rho: \Gal_{\Q_p}\ra \GL_3(\Q_p)$ such that $\rho$ has a modulo $p$ reduction equal to $\overline{\rho}$ and $D_{\rig}(\rho)=:D$ is given as  in \S~\ref{S1}. We use the notation there, in particular, $D$ is determined by $\ul{\phi}$, $h$, and $\iota_D: D_1 \hookrightarrow C_1$. Let $\fm_{\rho}\subset R_{\overline{\rho}}^{\square}[1/p]$ be the maximal ideal associated to $\rho$, and suppose there exists a maximal ideal $\fm^{\wp}$ of $R_{\infty}^{\wp}[1/p]$ such that $\Pi_{\infty}[\fm]^{\lalg}\neq 0$ for $\fm=(\fm_{\rho}, \fm^{\wp})$, the corresponding maximal ideal of $R_{\infty}[1/p]$. By \cite[\S~4]{CEGGPS1}, we have $\Pi_{\infty}[\fm]^{\lalg} \cong \pi_{\alg}(\ul{\phi}, \lambda)$. 
	This implies that for any refinement $w(\ul{\phi})$,   $x_w:=(x_{w,\wp}, \fm^{\wp})=(\rho, \jmath(w(\ul{\phi}))\delta_Bz^{\lambda}, \fm^{\wp})\in \cE$. By Hypothesis \ref{Hnonc}, all the points are non-critical. In particular, $X_{\tri}^{\square}(\overline{\rho})$ is smooth at the points $x_{w,\wp}$ (cf. \cite[Thm.~2.6 (iii)]{BHS1}). As $(\Spf R_{\infty}^{\wp})^{\rig}$ is also smooth at $\fm^{\wp}$, $\cE$ is smooth at all $x_{w}$. By  \cite[Lem.~3.8]{BHS2} and the multiplicity one property in the construction in \cite{CEGGPS1},  we see  $\cM$ is locally free of rank one at all $x_w$.  Let $\cI:=(\fm_{\rho}^2, \fm^{\wp})\subset R_{\infty}[1/p]$. 

Let $\kappa_3$ denote the composition $\cE \ra \widehat{T} \xrightarrow{\pr_3} \widehat{\Q_p^{\times}}$, where $\pr_3$ denotes the projection induced by the restriction map $\chi\mapsto \chi|_{\diag(1,1,\Q_p^{\times})}$. 
Let $\overline{\cE}$ be the fibre of $\cE$ at $\phi_3 z^{h_3}|\cdot|^{-2}$ via $\kappa_3$. Thus $\overline{\cE}$ consists of points of the form $(\rho', \delta_1' \boxtimes \delta_2' \boxtimes \phi_3 z^{h_3}|\cdot|^{-2}, \fm'^{\wp})$. We first show a local-global compatibility result for the tangent space of $\overline{\cE}$ at the point $x:=x_1$ (which clearly generalizes to $x_w$ for $w\in S_3$). One may obtain similar results for the tangent space of $\cE$ at $x$, but the results for $\overline{\cE}$ are enough for our application.

	The embedding $\cE \hookrightarrow (\Spf R_{\overline{\rho}}^{\square})^{\rig} \times \widehat{T} \times (\Spf R_{\infty}^{\wp})^{\rig}$ induces  $\overline{\cE} \hookrightarrow (\Spf R_{\overline{\rho}}^{\square})^{\rig} \times \widehat{T}_1\times (\Spf R_{\infty}^{\wp})^{\rig}$.  Let $\cU$ be an affinoid smooth neighbourhood of $x$ in $\overline{\cE}$ of the form $\cU_{\wp} \times \cU^{\wp}$ where $\cU_{\wp}\subset  (\Spf R_{\overline{\rho}}^{\square})^{\rig} \times \widehat{T}_1$ and $\cU^{\wp} \subset (\Spf R_{\infty}^{\wp})^{\rig}$. Let $\fm_{x, \wp}$ be the maximal ideal of $\co(\cU_{\wp})$ associated to $x_{\wp}$. Let $\cI_1:=(\fm_{x,\wp}^2, \fm^{\wp})$ (which is also the closed ideal generated by $\cI$ via $R_{\infty}[1/p]\ra \co(\cU)$). Denote by $\cM_{\tilde{x}}$ the fibre of $\cM$ at $\co(\cU)/\cI_1$ (which does not depend the choice of $\cU$), which is equipped with a natural action of $R_{\infty} \times T(\Q_p)$ (where $\diag(1,1,\Q_p^{\times})$ acts by $\phi_3 z^{h_3}|\cdot|^{-2}$). We have natural $R_{\infty} \times T(\Q_p)$-equivariant injections
	\begin{equation}\label{Efibres}
		\cM_{x}^{\vee} \hooklongrightarrow \cM_{\tilde{x}}^{\vee} \hooklongrightarrow J_B(\Pi_{\infty}^{R_{\infty}-\an}[\cI]).
	\end{equation}
	\begin{lemma}\label{Lbala0}
		The map $\cM_{\tilde{x}}^{\vee} \hookrightarrow J_B(\Pi_{\infty}^{R_{\infty}-\an}[\cI])$ is balanced.
	\end{lemma}
	\begin{proof}
		It follows by the same argument as in \cite[Lem.~4.11]{Ding7}.
	\end{proof}
	By \cite[Thm.~0.13]{Em2}, the maps in (\ref{Efibres}) induce $R_{\infty} \times \GL_3(\Q_p)$-equivariant injections:
	\begin{equation}\label{Einj1}
		I_{B^-}^{\GL_3} (\cM_{x}^{\vee}\delta_B^{-1}) \hooklongrightarrow I_{B^-}^{\GL_3}(\cM_{\tilde{x}}^{\vee}\delta_B^{-1})\hooklongrightarrow \Pi_{\infty}^{R_{\infty}-\an}[\cI].
	\end{equation} 
	\begin{lemma}\label{LJPinj}
		The map (\ref{Efibres}) induces balanced $R_{\infty} \times L_{P_1}(\Q_p)$-equivariant injections
		\begin{equation}\label{Eind2}
			(I_{B_1^-}^{\GL_2} \cM_x^{\vee} \delta_{B_1}^{-1}) \boxtimes \phi_3z^{h_3} |\cdot|^{-2} \hooklongrightarrow	(I_{B_1^-}^{\GL_2} (\cM_{\tilde{x}}^{\vee}\delta_{B_1}^{-1})) \boxtimes \phi_3 z^{h_3}|\cdot|^{-2} \hooklongrightarrow J_{P_1}(\Pi_{\infty}^{R_{\infty}-\an}[\cI]).
		\end{equation}
	\end{lemma}
	\begin{proof}
		For an admissible locally analytic representation $W$ of $T(\Q_p)$, using $$(\Ind_{B^-}^{\GL_3}W)^{\an} \cong \big(\Ind_{P_1^-}^{\GL_3} (\Ind_{B^-\cap L_{P_1}}^{L_{P_1}} W)^{\an}\big)^{\an},$$ we have by definition 
		$	I_{P_1^-}^{\GL_3} (I_{B^- \cap L_{P_1}}^{L_{P_1}} W) \xrightarrow{\sim} I_{B^-}^{\GL_3} W.
		$
		The lemma then follows from (\ref{Einj1}) and \cite[(0.10)]{Em2}. 
	\end{proof}
	We discuss the $R_{\infty}$-action. Let $R_{\rho}^{\square}$ be the universal  framed deformation ring of $\rho$ that is a formally smooth completed local $E$-algebra. Let $R_{\rho,\sF}^{\square}$ be the universal framed $\sF$-deformation ring, and $R_{\rho,\sF_1}^{\square}$ be the universal framed trianguline deformation ring of $\rho$ with respect to the refinement $\ul{\phi}$ (that we also denote by $\sF_1$). Denote by $\overline{R}_{\rho,\sF}^{\square}$ the universal framed deformation ring of $\sF$-deformations of the form $[\widetilde{D}_{1,A} \lin \cR_A(\phi_3 z^{h_3})]$, and $\overline{R}_{\rho, \sF_1}^{\square}$ the universal framed deformation ring of trianguline deformations with respect to $\sF_1$ of the form $[\widetilde{D}_{1,A} \lin \cR_A(\phi_3 z^{h_3})]$.
	We have a commutative diagram of formally smooth completed local $E$-algebras:
	\begin{equation*}
		\begin{CD}
			R_{\rho}^{\square} @>>>  \overline{R}_{\rho,\sF}^{\square}@>>> \overline{R}_{\rho,\sF_1}^{\square} \\ 
			@. @AAA @AAA \\
			@. R_{D_1} @>>> R_{D_1, \sF_1}
		\end{CD}
	\end{equation*}
	where all the horizontal maps are surjective. 
	Recall that the completion of $R_{\overline{\rho}}^{\square}[1/p]$ at $\rho$ is naturally isomorphic to $R_{\rho}^{\square}$, and consequently the completion of $X_{\tri}^{\square}(\overline{\rho})$ at $x_{\wp}$ is naturally isomorphic to $R_{\rho, \sF_1}^{\square}$ (e.g. see the proof of \cite[Thm.~2.6]{BHS1}). We see the $R_{\overline{\rho}}^{\square}$-action on $\cM_{\tilde{x}}$ factors through $\overline{R}_{\rho, \sF_1}^{\square}$ hence also factors through $\overline{R}_{\rho, \sF}^{\square}$. Moreover, $\cM_{\tilde{x}}$ is free of rank one over $\overline{R}_{\rho,\sF_1}^{\square}/\fm_{\rho}^2$, which  is thus isomorphic to $\cN_0 \otimes_{R_{D_1,\sF_1}/\fm_{D_1}^2} \overline{R}_{\rho,\sF_1}^{\square}/\fm_{\rho}^2$ for a free $R_{D_1, \sF_1}/\fm_{D_1}^2$-module $\cN_0$. We equip $\cN_0$ with the natural $T_1(\Q_p)$-action (induced by $ \Spf R_{D,\sF_1}\ra \widehat{T}_1$), then $\cN_0^{\vee}$ as $R_{D_1, \sF_1}\times T_1(\Q_p)$-module is no other than $\widetilde{\delta}_{\sF_1}^{\univ, \vee}$ in the precedent section. We have thus a $T_1(\Q_p)$-equivariant isomorphism of $\overline{R}_{\rho,\sF_1}^{\square}$-module:
	\begin{equation*}
		(\cN_0  \otimes_{R_{D_1,\sF_1}/\fm_{D_1}^2} \overline{R}_{\rho,\sF_1}^{\square}/\fm_{\rho}^2)\otimes_E (z^{-2} \boxtimes \varepsilon^{-1})\xlongrightarrow{\sim} 	\cM_{\tilde{x}}.
	\end{equation*}
	Consider the weight map 
	$$\kappa: \overline{\cE} \lra \widehat{T}_1 \lra \bA^2\xlongrightarrow{(a,b)\mapsto a-b} \bA^1,$$ and let $\cM_{\tilde{x}, g'}$ be the fibre of $\cM_{\tilde{x}}$ at   $h_1-h_2-1$. The action of $\overline{R}^{\square}_{\rho,\sF_1}$ on $\cM_{\tilde{x},g'}$ factors through $\overline{R}^{\square}_{\rho,g'}$, the universal deformation ring of deformations of the form $[\widetilde{D}_{1,A} \lin \cR_A(\phi_3z^{h_3})]$ with $\widetilde{D}_{1,A}$ de Rham up to twist by a certain character over $A$. We have  $\overline{R}^{\square}_{\rho,g'}\cong \overline{R}^{\square}_{\rho,\sF_i} \otimes_{R_{D_1,\sF_i}} R_{D_1,g'}$ and $\cM_{\tilde{x},g'}$ is a free $\overline{R}^{\square}_{\rho,g'}/\fm_{\rho}^2$-module. Letting $\cN_{0,g'}:=\cN_0 \otimes_{R_{D_1, \sF_1}} R_{D_1,g'}$, we have  an isomorphism $(\cN_{0,g'} \otimes_{R_{D_1,g'}/\fm_{D_1}^2} \overline{R}^{\square}_{\rho,g'}/\fm_{\rho}^2) \otimes_E (z^{-2} \boxtimes \varepsilon^{-1}) \xrightarrow{\sim}\ \cM_{\tilde{x},g'}$, and $\cN_{0,g'}^{\vee}\cong \widetilde{\delta}_{\sF_1,g'}^{\univ, \vee}$. In summary, we have $R_{D_1}\times T_1(\Q_p)$-equivariant maps:
	\begin{equation*}
		\begin{CD}
			\cN_{0}\otimes_E (z^{-2} \boxtimes \varepsilon^{-1}) @>>> \cM_{\tilde{x}} \\ 
			@VVV @VVV \\
			\cN_{0,g'} \otimes_E (z^{-2} \boxtimes \varepsilon^{-1}) @>>>  \cM_{\tilde{x}, g'} @>>> \cM_x
		\end{CD}
	\end{equation*}
	where the bottom composition is surjective.
	We then deduce:
	\begin{lemma}\label{LdR2}
		We have an $R_{D_1} \times \GL_2(\Q_p)$-equivariant commutative diagram
		\begin{equation*}
			\begin{CD}
				I_{B_1^-}^{\GL_2} (\cM_x^{\vee}\delta_{B_1}^{-1}) @>>>		I_{B_1^-}^{\GL_2} (\cM_{\tilde{x},g'}^{\vee}\delta_{B_1}^{-1}) @>>>	I_{B_1^-}^{\GL_2} (\cM_{\tilde{x}}^{\vee}\delta_{B_1}^{-1}) \\
				@V \sim VV 	@VVV @VVV\\
				(I_{\sF_1,0}) \otimes_E z^{-1} \circ \dett	@>>> 	I_{B^-_1}^{\GL_2} (\widetilde{\delta}_{\sF_1,g'}^{\univ} (\varepsilon^{-1}z^{-1} \boxtimes z^{-1}))@>>> 	I_{B^-_1}^{\GL_2} (\widetilde{\delta}_{\sF_1}^{\univ}(\varepsilon^{-1}z^{-1} \boxtimes z^{-1})).
			\end{CD}
		\end{equation*}
	\end{lemma}
	Denote by $J_{P_1}(\Pi_{\infty}^{R_{\infty}-\an}[\cI])_0$ the maximal subrepresentation of $J_{P_1}(\Pi_{\infty}^{R_{\infty}-\an}[\cI])$ on which the group  $Z_3:=\diag(1,1,\Q_p^{\times})$ acts by $\phi_3 z^{h_3}|\cdot|^{-2}$ and which, as $\GL_2(\Q_p)$-representation, is isomorphic to an extension of a certain copy of   $\pi_{\alg}(D_1) \otimes_E z^{-1}\circ \dett$ by a certain copy of  $\pi_{\alg}(D_1) \otimes_E z^{-1}\circ \dett$. The following proposition follows from the fact $\cM$ is locally free of rank $1$ at $x$.
	\begin{proposition}\label{PdR1}
		The morphism $I_{B_1^-}^{\GL_2} (\cM_{\tilde{x},g'}^{\vee}\delta_{B_1}^{-1}) \boxtimes \phi_3z^{h_3} |\cdot|^{-2} \ra J_{P_1}(\Pi_{\infty}^{R_{\infty}-\an}[\cI])$ factors through an $R_{\infty} \times \GL_2(\Q_p)$-equivariant isomorphism
		\begin{equation*}
			I_{B_1^-}^{\GL_2} (\cM_{\tilde{x},g'}^{\vee}\delta_{B_1}^{-1}) \xlongrightarrow{\sim } J_{P_1}(\Pi_{\infty}^{R_{\infty}-\an}[\cI])_0.
		\end{equation*}
	\end{proposition}
	It is clear that  $I_{B_1^-}^{\GL_2} (\cM_{\tilde{x}}^{\vee}\delta_{B_1}^{-1})[\fm_{\rho}]\cong I_{\sF_i} \otimes_E z^{-1} \circ \dett$, and we put 
	\begin{eqnarray*}
		\pi_{\tilde{x}}&:=&I_{B_1^-}^{\GL_2} (\cM_{\tilde{x}}^{\vee}\delta_{B_1}^{-1}) \oplus_{I_{\sF_i}\otimes_E z^{-1} \circ \dett} \big(\pi(D_1)\otimes_E z^{-1} \circ \dett\big), \\
		\pi_{\tilde{x},g'}&:=&I_{B_1^-}^{\GL_2} (\cM_{\tilde{x},g'}^{\vee}\delta_{B_1}^{-1}) \oplus_{\pi_{\alg}(D_1)\otimes_E z^{-1} \circ \dett} \big(\pi(D_1)\otimes_E z^{-1} \circ \dett\big), 
	\end{eqnarray*}where the amalgamated sums are both for the action of $R_{\infty}$ and $\GL_2(\Q_p)$. Let $\sE_{g',P_1}$ be the $\GL_2(\Q_p)$-subrepresentation of $J_{P_1}(\Pi_{\infty}^{R_{\infty}-\an}[\cI])[Z_3=\phi_3z^{h_3}|\cdot|^{-2}]$ generated by $\pi(D_1)\otimes_E z^{-1} \circ \dett \hookrightarrow J_{P_1}(\Pi_{\infty}^{R_{\infty}-\an}[\fm_{\rho}])[Z_3=\phi_3z^{h_3}|\cdot|^{-2}]$ and $J_{P_1}(\Pi_{\infty}^{R_{\infty}-\an}[\cI])_0$. By Lemma \ref{LdR2} and Proposition \ref{PdR1}, we have 
	\begin{proposition}We have an $R_{D_1} \times \GL_2(\Q_p)$-equivariant commutative diagram
			\begin{equation*}
			\begindc{\commdiag}[300] 
			\obj(0,0)[a]{$	\pi(D_1)\otimes z^{-1} $}
			\obj(3,0)[b]{$\widetilde{\pi}(D_1)^{\univ}_{g'}\otimes z^{-1} $}
			\obj(7,0)[c]{$\widetilde{\pi}(D_1)^{\univ}_{\sF_1}\otimes z^{-1}$}
			\obj(0,2)[d]{$	\pi(D_1) \otimes z^{-1} $}
			\obj(3,2)[e]{$	\pi_{\tilde{x},g'} $}
			\obj(7,2)[f]{$\pi_{\tilde{x}} $}
			\obj(12,2)[g]{$J_{P_1}(\Pi_{\infty}^{R_{\infty}-\an}[\cI])[Z_3=\phi_3z^{h_3}|\cdot|^{-2}]$}
			\mor{a}{b}{}[+1,6]
			\mor{b}{c}{}[+1,6]
			\mor{d}{a}{}[+1,9]
			\mor{e}{b}{}[+1,8]
			\mor{f}{c}{}[+1,8]
			\mor{d}{e}{}[+1,6]
			\mor{e}{f}{}[+1,6]
			\mor{f}{g}{}[+1,6]
			\enddc
		\end{equation*}
		Moreover, the composition $\pi_{\tilde{x},g'} \ra 	\pi_{\tilde{x}} \ra J_{P_1}(\Pi_{\infty}^{R_{\infty}-\an}[\cI])$ factors through an $R_{\infty} \times \GL_2(\Q_p)$-equivariant isomorphism $\pi_{\tilde{x},g'} \xrightarrow{\sim} \sE_{g',P_1}$. 
	\end{proposition}
	Consider the point $x_{s_1}$. By the same argument, the statement in the above proposition holds with $x$ replaced by $x_{s_1}$, and $\sF_1$ replaced by $\sF_2$. By taking amalgamated sum and using Proposition \ref{Puniv2}, we finally obtain
	\begin{corollary}\label{CLG}
		We have an $R_{D_1} \times \GL_2(\Q_p)$-equivariant commutative diagram
			\begin{equation*}
			\begindc{\commdiag}[300] 
			\obj(0,0)[a]{$	\pi(D_1)\otimes z^{-1} $}
			\obj(3,0)[b]{$\widetilde{\pi}(D_1)^{\univ}_{g'}\otimes z^{-1} $}
			\obj(7,0)[c]{$\widetilde{\pi}(D_1)^{\univ}\otimes z^{-1}$}
			\obj(0,2)[d]{$	\pi(D_1) \otimes z^{-1} $}
			\obj(3,2)[e]{$\sE_{g',P_1}$}
			\obj(7,2)[f]{$	\pi_{\tilde{x}} \oplus_{\sE_{g',P_1}} \pi_{\tilde{x}_{s_1}} $}
			\obj(12,2)[g]{$J_{P_1}(\Pi_{\infty}^{R_{\infty}-\an}[\cI])[Z_3=\phi_3z^{h_3}|\cdot|^{-2}]$}
			\mor{a}{b}{}[+1,6]
			\mor{b}{c}{}[+1,6]
			\mor{d}{a}{}[+1,9]
			\mor{e}{b}{}[+1,8]
			\mor{f}{c}{}[+1,8]
			\mor{d}{e}{}[+1,6]
			\mor{e}{f}{}[+1,6]
			\mor{f}{g}{}[+1,6]
			\enddc
		\end{equation*}
		Moreover, the map $j: (\pi_{\tilde{x}} \oplus_{\sE_{g',P_1}} \pi_{\tilde{x}_{s_1}}) \boxtimes \phi_3 z^{h_3} |\cdot|^{-2}\ra J_{P_1}(\Pi_{\infty}^{R_{\infty}-\an}[\cI])$ is balanced. 
	\end{corollary} 
	\begin{proof}
		It rests to show the map is balanced.  It is clear that  $$(\pi(D_1) \otimes z^{-1}) \boxtimes \phi_3 z^{h_3}|\cdot|^{-2} \hooklongrightarrow J_{P_1}(\Pi_{\infty}^{R_{\infty}-\an}[\fm])$$ comes from an injection $I_{P_1^-}^{\GL_3} \big(((\pi(D_1) \otimes z^{-1} )\boxtimes \phi_3 z^{h_3}|\cdot|^{-2}) \delta_{P_1}^{-1}\big)\hookrightarrow \Pi_{\infty}^{R_{\infty}-\an}[\fm]$.  Together with  (\ref{Einj1}), we see $\pi_{\widetilde{x}} \boxtimes \phi_3 z^{h_3} |\cdot|^{-2}\hookrightarrow J_{P_1}(\Pi_{\infty}^{R_{\infty}-\an}[\cI])$ comes from the corresponding map  $I_{P_1^-}^{\GL_3}((\pi_{\widetilde{x}} \boxtimes \phi_3 z^{h_3}|\cdot|^{-2})\delta_{P_1}^{-1}) \ra \Pi_{\infty}^{R_{\infty}-\an}[\cI]$. The same holds with $\pi_{\widetilde{x}}$ replaced by $\pi_{\widetilde{x}_{s_1}}$. It is then not difficult to see $j$ corresponds to $I_{P_1^-}^{\GL_3}\big(\big((\pi_{\tilde{x}} \oplus_{\sE_{g',P_1}} \pi_{\tilde{x}_{s_1}}) \boxtimes \phi_3 z^{h_3} |\cdot|^{-2}\big)\delta_{P_1}^{-1}\big)\ra \Pi_{\infty}^{R_{\infty}-\an}[\cI]$. So $j$ is balanced by \cite[(0.10)]{Em2}.  \end{proof}
	\begin{remark}\label{Rlocglob}
		Note that any extension of $\pi_{\alg}(D_1) \otimes_E z^{-1} \circ \dett$ by $\pi(D_1)\otimes_E z^{-1} \circ \dett\hookrightarrow J_{P_1}(\Pi_{\infty}^{R_{\infty}-\an}[\fm_{\rho}])[Z_3=\phi_3z^{h_3}|\cdot|^{-2}]$, which is contained in $J_{P_1}(\Pi_{\infty}^{R_{\infty}-\an}[\fm_{\rho}])[Z_3=\phi_3 z^{h_3} |\cdot|^{-2}]$,  lies in $	\pi_{\tilde{x}} \oplus_{\sE_{g',P_1}} \pi_{\tilde{x}_{s_1}} $. Indeed, let $V$ be such an extension. As $	\pi_{\tilde{x}} \oplus_{\sE_{g',P_1}} \pi_{\tilde{x}_{s_1}} $ contains the universal extension of $\pi_{\alg}(D_1) \otimes_E z^{-1} \circ \dett$ by $\pi(D_1)\otimes_E z^{-1} \circ \dett$ (by the above commutative diagram), amalgamating $V$ with certain subrepresentation of $	\pi_{\tilde{x}} \oplus_{\sE_{g',P_1}} \pi_{\tilde{x}_{s_1}} $, we reduce to the case where $V$ is split. But in this case, $V$ is already contained in $\sE_{g',P_1}$. Consequently, for such an extension, the induced $R^{\square}_{\rho}/\fm_{\rho}^2$-action factors through $\overline{R}^{\square}_{\rho, \sF}$. 
	\end{remark}
Replacing the points $\{x, x_{s_1}\}$ by $\{x_{s_1s_2}, x_{s_2s_1s_2}\}$, $P_1$ by $P_2$, $Z_3$ by $Z_1:=\diag(\Q_p^{\times},1,1)$ and using the same arguments, we get
	\begin{corollary}
		We have an $R_{C_1}\times \GL_2(\Q_p)$-equivariant commutative diagram
					\begin{equation*}
			\begindc{\commdiag}[300] 
			\obj(0,0)[a]{$	\pi(C_1) \otimes |\cdot|^{-1} $}
			\obj(3,0)[b]{$ \widetilde{\pi}(C_1)^{\univ}_{g'}  \otimes |\cdot|^{-1}$}
			\obj(7,0)[c]{$\widetilde{\pi}(C_1)^{\univ} \otimes |\cdot|^{-1}$}
			\obj(0,2)[d]{$	\pi(C_1) \otimes |\cdot|^{-1}$}
			\obj(3,2)[e]{$	\sE_{g', P_2}$}
			\obj(7,2)[f]{$	\pi_{\tilde{x}_{s_1s_2}} \oplus_{\sE_{g',P_2}} \pi_{\tilde{x}_{s_2s_1s_2}} $}
			\obj(12,2)[g]{$J_{P_2}(\Pi_{\infty}^{R_{\infty}-\an}[\cI])[Z_1=\phi_3z^{h_1-2}]$}
			\mor{a}{b}{}[+1,6]
			\mor{b}{c}{}[+1,6]
			\mor{d}{a}{}[+1,9]
			\mor{e}{b}{}[+1,8]
			\mor{f}{c}{}[+1,8]
			\mor{d}{e}{}[+1,6]
			\mor{e}{f}{}[+1,6]
			\mor{f}{g}{}[+1,6]
			\enddc
		\end{equation*}
		Moreover, the map $\phi_3 z^{h_1-2} \boxtimes  \big(	\pi_{\tilde{x}_{s_1s_2}} \oplus_{\sE_{g',P_2}} \pi_{\tilde{x}_{s_2s_1s_2}} \big) \ra J_{P_2}(\Pi_{\infty}^{R_{\infty}-\an}[\cI])$ is balanced. 
	\end{corollary}
	\subsection{Surplus locally algebraic constituents and local-global compatibility}
	We prove our main local-global compatibility result. We keep the notation of the precedent section.
	
	\begin{lemma}\label{Lkey1}
		For any extension $\pi \in \Ext^1_{\sF}(\pi_{\alg}(\ul{\phi},\lambda), \pi_1(\ul{\phi},\lambda))$ or $\pi \in \Ext^1_{\sG}(\pi_{\alg}(\ul{\phi},\lambda), \pi_1(\ul{\phi},\lambda))$, $\pi$ is not a subrepresentation of $\Pi_{\infty}[\fm]$.
	\end{lemma}	
	\begin{proof}
		We only prove the case for $\sF$, the case for $\sG$ being the same. Suppose $\pi \hookrightarrow \Pi_{\infty}[\fm]$. 
		By \cite[Lem.~4.16]{BHS2},  for all $w\in S_3$, we have (``$\{-\}$" denoting the corresponding generalized eigenspace)
		\begin{equation*}
			\dim_E J_B(\Pi_{\infty}^{R_{\infty}-\an}[\fm])[T(\Q_p)=\jmath(w(\ul{\phi}))\delta_B z^{\lambda}]=\dim_E J_B(\Pi_{\infty}^{R_{\infty}-\an}[\fm])\{T(\Q_p)=\jmath(w(\ul{\phi}))\delta_B z^{\lambda}\}=1.
		\end{equation*}
		Together with Proposition \ref{PExt2} (3) and Remark \ref{RGL3tri}, we deduce $\pi\notin \Ext^1_{\sF_i}(\pi_{\alg}(\ul{\phi},\lambda), \pi_1(\ul{\phi},\lambda))$. It is clear (by (\ref{Efern1}), Proposition \ref{PExt3})
		\begin{equation*}
			\Ext^1_{1}(\pi_{\alg}(\ul{\phi},\lambda), \pi_1(\ul{\phi},\lambda))+\Ext^1_{s_1}(\pi_{\alg}(\ul{\phi},\lambda), \pi_1(\ul{\phi},\lambda))=\Ext^1_{\sF}(\pi_{\alg}(\ul{\phi},\lambda), \pi_1(\ul{\phi},\lambda)).
		\end{equation*}
		There exists hence $\pi'\in \Ext^1_{1}(\pi_{\alg}(\ul{\phi},\lambda), \pi_1(\ul{\phi},\lambda))$ such that $$[\pi'']:=[\pi]-[\pi']\in \Ext^1_{s_1}(\pi_{\alg}(\ul{\phi},\lambda), \pi_1(\ul{\phi},\lambda)).$$
		By  Remark \ref{RGL3tri},  there exists a deformation $\widetilde{\chi}$ of $\jmath(\ul{\phi})z^{\lambda} \delta_B$ (resp. $\widetilde{\chi}'$ of $\jmath(s_1(\ul{\phi}))z^{\lambda} \delta_B$) such that 
		\begin{equation*}
			\widetilde{\chi}\hooklongrightarrow J_B(\pi') \text{ (resp. } \widetilde{\chi}' \hooklongrightarrow J_B(\pi'')).
		\end{equation*}
		Let $v$ be a non-zero element in the tangent space of $\cE$ at $x$  such that the associated character of $T(\Q_p)$ is $\widetilde{\chi}$, and $\cI_v$ be its associated ideal of $R_{\infty}[1/p]$. Let $\widetilde{D}_v$ be the associated deformation of $D$. Then $\widetilde{\chi}^{\sharp}:=\widetilde{\chi}\delta_B^{-1}(\varepsilon^2 \boxtimes \varepsilon \boxtimes 1)$ is a trianguline parameter of $\widetilde{D}_v$. Note  $\widetilde{\chi}^{\sharp}$ has the form $\phi_1z^{h_1}(1+\psi_1\epsilon)\boxtimes \phi_2 z^{h_2}(1+\psi_2 \epsilon) \boxtimes \phi_3 z^{h_3}(1+(\psi_3)\epsilon)$. Using an easy variation of (\ref{Einj1}) for $\cE$ instead of $\overline{\cE}$ (see also the proof of \cite[Prop.~C.5]{Ding15}), we have $\pi'\hookrightarrow \Pi_{\infty}^{R_{\infty}-\an}[\cI_v]$. As $\pi \in \Pi_{\infty}^{R_{\infty}-\an}[\cI_v]$, we deduce $\pi''\in \Pi_{\infty}^{R_{\infty}-\an}[\cI_v]$, hence $\widetilde{\chi}'\hookrightarrow J_B(\Pi_{\infty}^{R_{\infty}-\an}[\cI_v])$. By construction of $\cE$, we see $(\cI_v, \widetilde{\chi}')\hookrightarrow \cE$. But this implies $(\widetilde{\chi}')^{\sharp}$ is a trianguline parameter of $\widetilde{D}_v'$ where $\widetilde{D}_v' \cong \widetilde{D}_v$ as self-extension of $D$. Note $(\widetilde{\chi}')^{\sharp}$ has the form $\phi_2z^{h_1}(1+\psi_2'\epsilon)\boxtimes \phi_1 z^{h_2}(1+\psi_1' \epsilon) \boxtimes \phi_3 z^{h_3}(1+\psi'_3\epsilon)$. Using Proposition  \ref{PGL21} (2), it is not difficult to see  there exists $a\in E^{\times}$ such that $\psi_3'=a\psi_3$,  and  $\psi_1-\psi_2, \psi_1'-\psi_2' \in \Hom_{\sm}(\Q_p^{\times},E)$. By Remark \ref{RGL3tri}, both $\pi'$ and $\pi''$ lie in $\Ext^1_1(\pi_{\alg}(\ul{\phi},\lambda), \pi_1(\ul{\phi},\lambda))$, hence so does $\pi$, a contradiction.
	\end{proof}
	\begin{theorem}\label{Tmain2}
		For an injection $\iota\in \Hom_{(\varphi, \Gamma)}(D_1, C_1)$, $\pi(\ul{\phi},\lambda, \iota)\hookrightarrow \Pi_{\infty}^{R_{\infty}-\an}[\fm]$ if and only if $\iota=\iota_D$.
	\end{theorem}
	\begin{proof}
		Let $(\widetilde{D}_1, \widetilde{C}_1)\in \sI_D^0$ and let $\widetilde{D}$ be an associated deformation of $D$ with $\kappa_{\sF}(\widetilde{D})=(\widetilde{D}_1,0)$ and $\kappa_{\sG}(\widetilde{D})=(\widetilde{C}_1, 0)$ (as in Theorem \ref{ThIW1}). Let $v$ be a non-zero  element in the tangent space of $R_{\rho,\sF}^{\square}$ associated to $\widetilde{D}$, $\cI_{v,\wp} \subset R_{\overline{\rho}}^{\square}[1/p]$ be the associated ideal, and $\cI_v:=\cI_{v,\wp}+\fm^{\wp}\subset R_{\infty}[1/p]$ (which corresponds to the element $(v,0)$ in the tangent space of $(\Spf R_{\overline{\rho}}^{\square})^{\rig} \times (\Spf R_{\infty}^{\wp})^{\rig}$ at $\fm$). We have a natural exact sequence (associated to the non-zero $v$)
		\begin{equation*}
			0 \lra \Pi_{\infty}^{R_{\infty}-\an}[\fm] \lra \Pi_{\infty}^{R_{\infty}-\an}[\cI_v] \xlongrightarrow{\kappa} \Pi_{\infty}^{R_{\infty}-\an}[\fm].
		\end{equation*}
		
		\noindent \textbf{Claim:} $(\pi_{\tilde{x}} \oplus_{\sE_{g',P_1}} \pi_{\tilde{x}_{s_1}})[\cI_v]\cong \rec(\widetilde{D}_1) \otimes_E z^{-1}\circ \dett$.
		
		\noindent \textit{Proof of the claim.} Let $\cJ_v \subset R_{D_1}$ be the kernel of the natural map $R_{D_1} \ra R_{\overline{\rho}}^{\square}[1/p]/\cI_{v,\wp}$.  The right vertical map in Corollary \ref{CLG} induces $f: (\pi_{\tilde{x}} \oplus_{\sE_{g',P_1}} \pi_{\tilde{x}_{s_1}})[\cI_v] \ra \widetilde{\pi}(D_1)^{\univ}[\cJ_v] \otimes_E z^{-1} \circ \dett\cong \rec(\widetilde{D}_1)\otimes_E z^{-1} \circ \dett$. Note the kernel of the map is a finite copy of $\pi_{\alg}(D_1) \otimes_E z^{-1} \circ \dett$. If $f$ is not an isomorphism,  $(\pi_{\tilde{x}} \oplus_{\sE_{g',P_1}} \pi_{\tilde{x}_{s_1}})[\cI_v] $ has to contain $(\pi_{\alg}(D_1) \oplus \pi(D_1)) \otimes_E z^{-1} \circ \dett$. Applying the  Jacquet-Emerton functor $J_{B\cap L_{P_1}}(-)$, and using \cite[Lem.~4.16]{BHS2} and  similar arguments as in the proof of Lemma \ref{Lkey1}, this implies that $\widetilde{D}$ is trianguline, a contradiction. The claim follows.

		We obtain hence an injection $(\rec(\widetilde{D}_1)\otimes_E z^{-1} \circ \dett) \boxtimes (\phi_3 z^{h_3}|\cdot|^{-2}) \hookrightarrow J_{P_1}(\Pi_{\infty}^{R_{\infty}-\an}[\cI_v])$, which is balanced by Corollary \ref{CLG}. By \cite[Thm.~0.13]{Em2} (noting similarly as in the discussion below Proposition \ref{PExt1}, $j^-$ can also be obtained by applying Emerton's $I_{P_1^-}^{\GL_3}(-)$), it  induces  $i^-(\widetilde{D}_1)\hookrightarrow \Pi_{\infty}^{R_{\infty}-\an}[\cI_v]$. Note that the composition 
		\begin{equation*}
			i^-(\widetilde{D}_1) \hooklongrightarrow \Pi_{\infty}^{R_{\infty}-\an}[\cI_v] \xlongrightarrow{\kappa} \Pi_{\infty}^{R_{\infty}-\an}[\fm]
		\end{equation*} 
		has image exactly equal to $\pi_{\alg}(\ul{\phi},\lambda)$. Indeed, the inclusion is clear. However,  the composition can not have zero image by Lemma \ref{Lkey1}. Similarly, we have $i^+(\widetilde{C}_1)\hookrightarrow \Pi_{\infty}^{R_{\infty}-\an}[\cI_v]$ whose composition with $\kappa$ has image equal to $\pi_{\alg}(\ul{\phi},\lambda)$ as well (noting $v$ can also be viewed as an element in the tangent space of $R_{\rho, \sG}^{\square}$ by Theorem \ref{ThIW1}).
		So there is an extension $\pi\subset i^-(\widetilde{D}_1) \oplus_{\pi_1(\ul{\phi},\lambda)} i^+(\widetilde{C}_1)$ of $\pi_{\alg}(\ul{\phi}, \lambda)$ by $\pi_1(\ul{\phi},\lambda)$ such that $\pi\hookrightarrow \Pi_{\infty}^{R_{\infty}-\an}[\fm]$. 
		
		We show $\pi$ has to be isomorphic to $\pi(\ul{\phi},\lambda, \iota_D)$. As in the proof of Proposition \ref{Pinter}, let $(x_0,y_0)\in \sI_D$ be non-zero de Rham, and $([\widetilde{D}_1']=[\widetilde{D}_1]+x_0, [\widetilde{C}_1']=[\widetilde{C}_1]+y_0)\in \sI_D^0$. Let $v'$, $\widetilde{D}'$, $\cI_{v'}$ be the similar objects associated to the pair. By the same argument, $i^-(\widetilde{D}_1')$, $i^+(\widetilde{C}_1')$ are both subrepresentations of $ \Pi_{\infty}^{R_{\infty}-\an}[\cI_{v'}]$. 
		We also know $\pi\subset \Pi_{\infty}^{R_{\infty}-\an}[\fm]\subset  \Pi_{\infty}^{R_{\infty}-\an}[\cI_{v'}]$. If $\pi$ is not isomorphic to $\pi(\ul{\phi},\lambda, \iota_D)$,  $[\pi]$ has the form $a i^-([\widetilde{D}_1])-b i^+([\widetilde{C}_1])$ with $a\neq b$. Moreover, by Proposition \ref{Pinter} (and the proof), $\pi$ is not contained in $i^-(\widetilde{D}_1') \oplus_{\pi_1(\ul{\phi},\lambda)} i^+(\widetilde{C}_1')$, we have hence an injection
		\begin{equation*}
			\pi \oplus_{\pi_1(\ul{\phi},\lambda)} i^-(\widetilde{D}_1') \oplus_{\pi_1(\ul{\phi},\lambda)} i^+(\widetilde{C}_1') \hooklongrightarrow \Pi_{\infty}^{R_{\infty}-\an}[\cI_{v'}].
		\end{equation*}
Using 	$i^-([x_0])=i^+([y_0])$ (for example see the proof of Lemma \ref{Lindep2}), it is not difficult to deduce   (the representation corresponding to) $i^-([x_0])$ injects into $\Pi_{\infty}^{R_{\infty}-\an}[\cI_{v'}]$.
 However, applying $J_B(-)$ and (again) using \cite[Lem.~4.16]{BHS2} and similar arguments in the proof of Lemma \ref{Lkey1}, this will imply  that $\widetilde{D}'$ is trianguline, a contradiction.

We prove the  ``only if" part of the theorem. 	Suppose there exists an injection $\iota\notin E[\iota_D]$ such that $\pi(\ul{\phi},\lambda, \iota)\hookrightarrow \Pi_{\infty}^{R_{\infty}-\an}[\fm]$. Recall the extension class $[\pi(\ul{\phi},\lambda, \iota)]$ has the form $i^-([\widetilde{D}_1])-i^+([\widetilde{C}_1])$ with $(\widetilde{D}_1,\widetilde{C}_1)\in \sI_{\iota}^0$. By Proposition \ref{Ppairing}, $\widetilde{D}_1\notin \Ext^1_{\iota_D}(D_1,D_1)$ and $\widetilde{C}_1\notin \Ext^1_{\iota_D}(C_1,C_1)$.  Let $\widetilde{D}$ be a deformation of $D$ such that $\kappa_{\sG}(\widetilde{D})=(\widetilde{C}_1,0)$ (cf. Remark \ref{Rkappf}), and $v$, $\cI_v$ be associated to $\widetilde{D}$ similarly as above. By the same argument (using the claim with $D_1$ replaced by $C_1$), we have $i^+(\widetilde{C}_1)\hookrightarrow \Pi_{\infty}^{R_{\infty}-\an}[\cI_v]$. As $\pi(\ul{\phi},\lambda, \iota)\hookrightarrow \Pi_{\infty}^{R_{\infty}-\an}[\fm]\subset  \Pi_{\infty}^{R_{\infty}-\an}[\cI_v]$, we deduce $i^-(\widetilde{D}_1)\hookrightarrow  \Pi_{\infty}^{R_{\infty}-\an}[\cI_v]$ hence $\rec(\widetilde{D}_1)\otimes_E z^{-1} \circ \dett \hookrightarrow J_{P_1} (\Pi_{\infty}^{R_{\infty}-\an}[\cI_v])[Z_3=\phi_3 z^{h_3}|\cdot|^{-2}]$. By Remark \ref{Rlocglob}, the  $R^{\square}_{\rho}$-action on $i^-(\widetilde{D}_1)$ factors through $\overline{R}^{\square}_{\rho, \sF}$, hence so does its action on $i^+(\widetilde{C}_1)$ (whose extension class is equal to $[\pi(\ul{\phi}, \lambda, \iota)]+i^-([\widetilde{D}_1])$, noting the $R^{\square}_{\rho}$-action on $\pi(\ul{\phi}, \lambda, \iota)$ factors through $R^{\square}_{\rho}/\fm_{\rho}$). However, as $\widetilde{C}_1\notin \Ext^1_{\iota_D}(C_1,C_1)$ hence $\widetilde{D}\notin \Ext^1_{\sF}(D,D)$ (Theorem \ref{ThIW1}), we see the image of $\cI_{v,\wp}$ in $\overline{R}^{\square}_{\rho,\sF}$ is $\fm$. Thus $i^+(\widetilde{C}_1)$ ($\subset \Pi_{\infty}^{R_{\infty}}[\cI_v]$) has to be  annihilated by $\fm$, contradicting Lemma \ref{Lkey1}.
	\end{proof}

\end{document}